\documentclass{amsart}
\usepackage{graphicx}
\usepackage{amssymb}
\usepackage{amsfonts}
\usepackage{hyperref}
\usepackage{mathrsfs}
\swapnumbers
\newtheorem{theorem}{Theorem}[section]
\newtheorem{lemma}[theorem]{Lemma}
\newtheorem{corollary}[theorem]{Corollary}
\newtheorem{proposition}[theorem]{Proposition}
\theoremstyle{definition}
\newtheorem{definition}[theorem]{Definition}
\newtheorem{assumption}[theorem]{Assumption}
\newtheorem{remark}[theorem]{Remark}
\newtheorem{example}[theorem]{Example}

\numberwithin{equation}{section}
\theoremstyle{plain}

\numberwithin{equation}{section} 
\numberwithin{figure}{section} 
\theoremstyle{plain}
\theoremstyle{plain}
\theoremstyle{remark}
\newtheorem*{acknowledgement*}{Acknowledgement}
\theoremstyle{example}

\usepackage[margin=3cm]{geometry}

\newcommand{\cA}{{\mathcal A}}
\newcommand{\cB}{{\mathcal B}}

\newcommand{\cD}{{\mathcal D}}

\newcommand{\cF}{{\mathcal F}}
\newcommand{\cG}{{\mathcal G}}

\newcommand{\cJ}{{\mathcal J}}

\newcommand{\cL}{{\mathcal L}}

\newcommand{\cX}{{\mathcal X}}
\newcommand{\cY}{{\mathcal Y}}

\newcommand{\te}{{\theta}}

\newcommand{\om}{{\omega}}
\newcommand{\ve}{{\varepsilon}}
\newcommand{\del}{{\delta}}

\newcommand{\sig}{{\sigma}}
\newcommand{\al}{{\alpha}}

\newcommand{\ka}{{\kappa}}
\newcommand{\la}{{\lambda}}

\newcommand{\bbC}{{\mathbb C}}
\newcommand{\bbE}{{\mathbb E}}

\newcommand{\bbN}{{\mathbb N}}
\newcommand{\bbP}{{\mathbb P}}
\newcommand{\bbR}{{\mathbb R}}

\newcommand{\bbZ}{{\mathbb Z}}

\begin{document}
\title[]{Statistical properties of Markov shifts: part II-LLT}  
 \vskip 0.1cm
 \author{Yeor Hafouta}
\address{
Department of Mathematics, The University of Florida}
\email{yeor.hafouta@mail.huji.ac.il}%

\thanks{ }
\dedicatory{  }
 \date{\today}

\maketitle
\markboth{Y. Hafouta}{ } 
\renewcommand{\theequation}{\arabic{section}.\arabic{equation}}
\pagenumbering{arabic}

\begin{abstract}
We prove Local (Central) Limit Theorems (LLT)  for partial sums of the form $S_n=\sum_{j=0}^{n-1}f_j(...,X_{j-1},X_j,X_{j+1},...)$, where $(X_j)$ is an exponentially fast mixing Markov chain with equicontinuous  conditional probabilities satisfying some ``physicality" assumptions
and $f_j$ are equicontinuous functions. Our conditions will always be in force when the chain takes values on a metric space and have uniformly bounded away from $0$ backward transition densities with respect to a measure which assigns uniform positive mass to certain ``balls".
 This paper complements \cite{MarShif1} where Berry-Esseen theorems, were proven for (not necessarily continuous) functions satisfying certain approximation conditions. Our results address a question posed by D. Dolgopyat and O. Sarig in \cite[Section 1.5]{DS}.

 In order to determine  the type of local limit theorem that holds, we need to develop a reduction theory in the sense of \cite{DolgHaf LLT, DS} for equicontinuous functions $f_j$, which allows us to determine whether ``$f_j$ can essentially be reduced to integer valued functions".  
As a byproduct of our methods we are also able to prove first order expansions in the irreducible case, which is crucial in detemining when better than the general optimal $O(\|S_n\|^{-1})$ central limit theorem rates can be achieved.  Even though the case of non-stationary sequences and time dependent functions $f_j$ is more challenging, our results seem to be new already for stationary Markov chains and a sequence of functions $f_j$ (or even for a single function $f_j=f$). They also seem to be new for non-stationary Bernoulli shifts (that is when $(X_j)$ are independent but not identically distributed).

Like in \cite{MarShif1}, our results apply to  Markov shifts in random dynamical environment, 
products of random non-stationary positive matrices and other operators, 
random Lyapunov exponents as well as several processes arising in statistics and applied probability like (some classes of) linear processes and inhomogeneous  iterated random functions. Most of these examples seem to only be treated in literature for iid $X_j$ and here we are able to drop both the stationarity and the independence assumptions. 
Like in \cite{MarShif1}, our proofs are based on conditioning on the future instead of the regular conditioning on the past that is used to obtain similar results when $f_j(...,X_{j-1},X_j,X_{j+1},...)$ depends only on $X_j$ (or on finitely many variables). In particular we generalize the LLT together with the reduction theory in \cite{DS} to functions which depend on the entire path of the chain, and the local limit theorems  in \cite{DolgHaf LLT} (when restricting to the Markovian case) to more general chains than the finite state ones considered there. 

Our results significant for both practitioners from statistics and applied probability and theorists in probability theory, ergodic theory and dynamical systems. For instance, we generalize \cite{HennionAoP97} from iid matrices to non-stationary Markovian ones (under suitable conditions) and get local limit theorems in the (non-stationary) setup of \cite{KestFurs61}.  We expect many other applications of our abstract results, for instance, to Markovian inhomogeneous random walks on $\text{GL}_d(\bbR)$, but in order not to overload the paper this will be discussed in future works.
\end{abstract}

\section{Introduction}
\label{SSIntro}

Limit theorems for weakly dependent sequences of random variables have been studied extensively in the last century. One of the main result is the Central Limit Theorem (CLT). For stationary sequences the CLT  states that the probability that a partial sum at time $N$
 belongs to an interval
of size $\sqrt{N}$ is asymptotically
the same as for the normal random variable with the same mean and variance.
In many problems one needs to see if the same conclusion holds for unit size intervals.  Such results follow from the local (central) limit theorem (LLT), which has applications to various areas of mathematics including mathematical physics
\cite{BLRB, DN16, DN-Mech, DT77, EM22, Kh49, LPRS, PS23}, number theory \cite{ADDS, Beck10, DS17}, 
geometry (\cite{LL15}), PDEs \cite{HKN}, and combinatorics (\cite{Can76,BR83, GR92, Hw98}).
Applications to dynamics include  abelian covers (\cite{BFRT, CaRa, DNP, OP19}),
suspensions flows (\cite{DN-llt}), skew products (\cite{Br05, DSL16, DDKN1, DDKN2, LB06}), and homogeneous dynamics (\cite{BeQu, BRLLT05, BRLLT23, Guiv2015, Hough2019}).
Other areas  of interest for non-autonomous
local limit theorems are applications to renormalization
(cf.  \cite{ARHW22, ARHW23, ADDS, BU}) and to random walks in random environment
(\cite{BCR16, CD24, DG13, DG21, DK20}).

Traditionally, in literature most results concerning limit theorems are obtained for stationary sequences, which can be viewed as an autonomous dynamical system generated by a single deterministic map preserving the probability law generated by the path of the process. One of the current challenges in the field of stochastic processes and dynamical systems  is to better understand non-stationary processes, namely
random and time-varying dynamical systems, in particular to develop novel probabilistic techniques to prove limit theorems. This direction of research, the ambition of which is to approach more the real by taking in account a time dependence inherent\footnote{e.g. external forces affect the local laws of physics, the uncertainty principle etc.} in some phenomena, has recently seen an enormous amount of activity. Many difficulties and questions emerge from this non-stationarity and time dependence. Let us mention for example the existence of many open questions about the establishment of quenched and sequential limit theorems  for systems with random or non-autonomous dynamics. The study of these systems opens new interplays between probability theory and dynamical systems, and leads to interesting insights in other areas of science. 

Recently there has been significant interest in limit theorems of
 non-stationary systems. 
In particular, the CLT for non autonomous hyperbolic systems was investigated in 
\cite{ANV, ALS09, Bk95,  CLBR, CR07, LD1, DFGTV, YH Adv, YH 23, HNTV, Kifer 1998, NSV12, NTV ETDS 18, CLT3, CLT2} and in  \cite{Cogburn, Dob, Kifer 1998, Pelig, SeVa}  for inhomogeneous Markov chains.
By contrast, the LLT has received much less attention and, it has been only
established for random systems under strong additional assumptions \cite{DPZ, DFGTV, DFGTV1, HK, Nonlin, YH YT}, see also \cite{BCR16, DS, MPP LLT1, MPP LLT2} for the results for  Markov chains. 
In \cite{DolgHaf LLT} we proved an LLT for non-stationary subshifts of finite type and related transformations, which have applications to finite state uniformly elliptic Markov shifts (namely, for functionals that depend on the entire path of the chain). This seems to be the only paper in literature when the LLT is obtained without special assumptions that ensure, for instance, that the variance of the sum grows linearly fast.
In fact, the question of the LLT is quite
subtle, and even in the setting of independent identically distributed (iid) random variables the local distribution 
depends on the arithmetic properties of the summands. If the summands do not take values in a lattice  then the local distribution is Lebesgue, and otherwise it is the counting measure on the lattice. The case when the local distribution is Lebesgue will be referred to as the non-lattice LLT while the case when the local distribution is an appropriate counting measure will be referred to as the lattice LLT. The exact definitions are postponed to \S \ref{Sec LLT}.

Concerning the LLT, for stationary Markov chains with countable state space it was studied for the first time in \cite{Nag57}. Since then the LLT has been established for a variety of stationary processes, mostly via spectral theory of quasi compact operators \cite{Jon, GH,HH, KM} or renewal theory \cite{GO} (which in its core also relies on quasi-compactness). The number of papers in the stationary case is enormous, and so we make no attempt in providing  a complete  list of references. The problem in proving the LLT in the non-stationary case is that there is no single operator to analyze, and instead there are sequences of operators, which make tools from ``soft" analysis like quasi-compactness unavailable.

In this paper for certain classes of inhomogeneous Markov  chains $(X_j)$ with sufficiently regular transition probabilities we prove a general LLT for partial sums of the form $S_n=\sum_{j=0}^{n-1}f_j(...,X_{j-1},X_j,X_{j+1},...)$ formed by an equicontinuous sequence of functions $f_j$. This is in contrast to \cite{BCR16, DS, MPP LLT1, MPP LLT2} where local limit theorems where obtained when $f_j$ depends only on finitely many variables.
In fact,
we identify the complete set obstructions to the non lattice LLT,
 see the discussion in \S \ref{SSObstructions}. 
In the autonomous case this is done by using a set of tools called
 Livsic theory (see \cite{Livsic}).
 In that case (see \cite{HH}) the non-lattice LLT fails only if the underlying function forming the partial sums is lattice valued, up to a coboundary. 
 In the autonomous case different notions of coboundary (measurable, $L^2$, continuous,
 H\"older, smooth) are equivalent, see \cite{PP, dLMM, Wil}, but this is 
  false in the non stationary setting (see \cite{BDH} or Theorem \ref{Var them}).
 In the course of the proof of our main results we develop a reduction theory for inhomogeneous Markov  shifts, generalizing the corresponding results for Markov  chains \cite{DS} and  chaotic non-autonomous dynamical systems \cite{DolgHaf LLT} (which have applications to finite state Markov shifts).  
We stress that we have no additional assumptions. In particular, we neither assume randomness of the Markov operator (i.e. the case of a random environment), nor that the variance of $S_n$ grows linearly in $n$. We also obtain a lattice LLT when the obstruction to the non-lattice LLT happen, and so we have no special irreducibility or reducibility assumptions on $(f_j)$.

Our results significant for both practitioners from statistics and applied probability and theorists in probability theory, ergodic theory and dynamical systems.
Like in \cite{MarShif1}, our results apply to  Markov shifts in random dynamical environments, 
products of random non-stationary positive matrices and other operators, 
random Lyapunov exponents as well as several processes arising in statistics and applied probability like (some classes of) linear processes and inhomogeneous  iterated random functions. Most of these examples seem to only be treated in literature for iid $X_j$ and here we are able to drop both the stationarity and the independence assumptions. 
 In particular, we generalize the LCLT in \cite{HennionAoP97} from iid matrices to non-stationary Markovian ones (under suitable conditions) and get local limit theorems in the (non-stationary) setup of \cite{KestFurs61}.  We expect many other applications of our abstract results, for instance, to Markovian inhomogeneous random walks on $\text{GL}_d(\bbR)$, but in order not to overload the paper this will be discussed in future works.

\subsection{Outline of the proof}
Like in \cite{MarShif1} our proofs are based on conditioning on the future instead of the past like in the classical case of functionals $f_j=f_j(X_j)$ of Markov chains $(X_j)$. In \cite{DolgHaf LLT} we proved an LLT for non-stationary subshifts of finite type and other related expanding maps. The proof in \cite{DolgHaf LLT} relies on the following ingredients. The first is that the iterates of the transfer operators satisfy an appropriate sequential Perron-Frobenius theorem with respect to appropriate norms. In Theorem \ref{RPF} we obtain such results with norms corresponding to ``graded" type of modulus of continuity of the transition operators and the functions $f_j$. The second ingredient is existence of two norms such that boundedness in the stronger norm and convergence to $0$ in $L^1$ of a sequence of functions implies convergence to $0$ in the weaker norm. In \cite{DolgHaf LLT} we worked with H\"older norms, where the strong norm corresponds to a H\"older norm with larger exponent. In our case the norms are a bit more complicated, and we take the weak norm to be the norm corresponding to a small power of the given modulus of continuity, see Lemma \ref{al be}. The third ingredient is a Lasota-Yorke inequality (or Doeblin-Fortet inequality), which we prove in Lemma \ref{Lasota Yorke}. Building on this inequality we are able to prove a reduction lemma, Lemma \ref{LmTDSmall},  which replaces the corresponding results in \cite[Lemmata 5.1 and 5.2]{DolgHaf LLT}. Roughly speaking this lemma shows that when the operator norm of the Fourier-Markov operator corresponding to the characteristic functions of the partial sums is close to one then the system enjoys some periodicity.
Using these tools we then apply the general mechanism developed in \cite{DolgHaf LLT} which includes a gluing lemma that allows us to combine successively the aforementioned periodicity and get periodicity along large blocks. Once this is established  the proof proceeds like in \cite[Sections 6 and 7]{DolgHaf LLT}, so the main novelty here is in  finding appropriate norms for which we have all the ingredients of the general scheme in \cite{DolgHaf LLT}. In particular, serious efforts have been made in order to avoid assuming that the functions $f_j$ are H\"older continuous.
Finally, to handle functions  $f_j=f_j(...,X_{j-1},X_j,X_{j+1},...)$ of the two sided path of the underlying Markov chain we prove a Sinai type lemma with respect to  norms we define, Lemma \ref{Sinai}, and then deduce the results for the two sided functions from the results for the one sided functions using conditioning on the past.

\section{Preliminaries and main results}
\subsection{The setup and standing assumptions}
Let $(X_j)_{j\geq0}$ be Markov chain such that $X_j$ takes values in a metric space $(\cX_j,d_j)$ whose diameter not exceed $1$.
Recall that the reverse $\phi$-mixing coefficients $\phi_R(n)$ a sequence $(X_j)$ of random variables are given by 
$$
\phi_R(n)=\sup_{k\in\bbN}\left\{\left|\bbP(A|B)-\bbP(A)\right|: A\in\cF_{0,k}, B\in\cF_{k+n,\infty}, \bbP(B)>0\right\}
$$
where $\cF_{u,v}$ is the $\sigma$-algebra generated by $X_s$ for all finite $u\leq s\leq v$. Our first standing assumption is the following mixing rates
\begin{assumption}\label{MixAss}
There exist constants $A>0$ and $\gamma\in(0,1)$ such that for all $n\in\bbN$,
$$
\phi_R(n)\leq A\gamma^n.
$$
\end{assumption}
\begin{remark}\label{Rem pi}
 Recall that the  classical $n$-th step (backward) Dobrushin contraction coefficients (\cite{Dob}) are given by   
 $$
\pi_j(n)=\sup_{x,y}\|\bbP(X_j\in\cdot|X_{j+n}=x)-\bbP(X_j\in\cdot|X_{j+n}=y)\|_{TV}.
 $$
Then by arguing like in the proof of \cite[Lemma 3.3]{HafSPA} for every $n_0\in\bbN$ there exists $C_{n_0}>0$ such that for all $n$
$$
\phi_R(n)\leq C_{n_0}\left(\sup_j\pi_{j}(n_0)\right)^{n/n_0}
$$
and so Assumption \ref{MixAss} is in force when $\sup_j\pi_j(n_0)<1$. Assumption \ref{MixAss} is also in force if the law of $X_{j+n_0}$ given $X_j$ has transition densities $p_{j,n_0}$ (with respect to some probability measure) which are uniformly bounded and bounded away from the origin. Indeed, in that case the chain is already exponentially fast $\psi$ mixing, see \cite[Ch. 4]{Brad}. 
\end{remark}

 Let us take a function $\om:[0,1]\to [0,1]$ which is continuous at the origin and $\omega(0)=0$. 
In what follows $\omega$ will play the role of a   modulus of continuity. Let us denote by $P_{j,n}(\cdot|x)$ the law of $(X_{j},...,X_{j+n-1})$ given $X_{j+n}=x$.
Our second standing assumption is that the measures $P_{j,n}(\cdot|x)$ are equicontinuous with respect to $x$ (with $\om$ being their modulus of continuity).
\begin{assumption}\label{Assss}
There exist constants $A_1>0$ and $n_0\in\bbN$ such that for all $j\geq0$ and $x,y\in\cX_{j+1}$ we have 
$$
\|P_{j,n_0}(\cdot|x)-P_{j,n_0}(\cdot|y)\|_{TV}\leq A_1\omega(d_{j+1}(x,y)).
$$
\end{assumption}

Our third standing assumption is that the measures $P_{j,n}$ are ``physical". Denote $\cX_{j,n}=\prod_{j\leq k<j+n}\cX_k$ and for $x,y\in\cX_{j,n}$,
$$
\om_{j,n}(x,y)=\max_{j\leq k<j+n}\om(d_j(x_j,y_j)).
$$

\begin{assumption}\label{Ball Ass}
There exists $n_0\in\bbN$ such that for all $r$ small enough,
$$
\zeta(r):=\inf\left\{P_{j,n}(\tilde B_{j,n_0}(b,r)|x): j\in\bbZ, b\in \cX_{j,n_0}, x\in\cX_{j+n_0}\right\}>0
$$ 
where $\tilde B_{j,n_0}(b,r)=\{y\in\cX_{j,n_0}: \om_{j,n_0}(b,y)\leq r\}$. 
\end{assumption}
\begin{remark}
Define a metric $\rho_{j,n}$ on $\cX_{j,n}$ by $\rho_{j,n}(x,y)=\max_{j\leq k<j+n}d_{k}(x_k,y_k)$. Let us denote by $B_{j,n}(x,r)$ a closed ball with radius $r$ around $x\in\cX_{j,n}$ with respect to this metric. Then if we assume that $\om$ is decreasing then
$$
B_{j,n}(x,\om^{-1}(r))\subseteq\tilde B_{j,n}(x,r).
$$
Thus Assumption \ref{Ball Ass} is weaker than requiring that  for every $s>0$ small enough we have
$$
\inf\left\{P_{j,n}(B_{j,n_0}(b,s)|x): j\in\bbZ, b\in \cX_{j,n_0}, x\in\cX_{j+n_0}\right\}>0.
$$
\end{remark}
\begin{example}
Assumptions \ref{MixAss},  \ref{Assss} and \ref{Ball Ass} are in force if $P_{j}(\Gamma|x)=\int_{\Gamma} p_j(x,y)d\nu_j(y)$ for a probability measure $\nu_j$ such that $\inf_{j,x}\nu_j(B_{j,1}(x,r))>0$ for all $r$ small enough and the transition densities $p_j(x,y)$ satisfy $\inf_{x,y,j}p_j(x,y)>0$ and $\sup_{x,y,j}p_j(x,y)<\infty$. 
\end{example}

Next, let $\kappa_j$ denote the law of $(X_j,X_{j+1},...)$. Set $\cX_{j,\infty}=\prod_{k\geq j}\cX_k$.
Take some $a\in(0,1)$ and for $x,y\in\cX_{j,\infty}$ define
$$
\om_j(x,y)=2A_1\sup_{n\geq 0}a^{n}\left(\omega(d_{j+n}(x_{j+n},y_{j+n}))\right)
$$
where $A_1$ comes from Assumption \ref{Assss}.
In what follows $\om_j$ will play the role of the modulus of continuity of functions on $\cX_{j,\infty}$.
\begin{remark}
   When $\cX_j$ are  discrete spaces   (i.e. $d_{j+n}(x_{j+n},y_{j+n}))=1$, unless $x_{j+n}=y_{j+n}$) then up to a constant $\omega_j$ coincides with the dynamical distance $\tilde\rho_j(x,y)=a^{\inf\{k: x_{j+k}\not=y_{j+k}\}}$.
\end{remark}
For  $g:\cX_{j,\infty}\to\bbR$ define a norm $\|g\|_{j,\omega}$ by
$$
\|g\|_{j,\om}=\|g\|_{L^\infty(\kappa_j)}+v_{j,\omega}(g)
$$
where 
$$
v_{j,\omega}(g)=\sup\left\{\frac{|g(x)-g(y)|}{\omega_j(x,y)}: x\not=y\right\}.
$$
\begin{remark}\label{RemA}
If     $v_{j,\omega}(g)<\infty$ then 
$$
\|g_j-\bbE[g_j|X_j,...,X_{j+r}]\|_{L^\infty}\leq 2A_1v_{j,\omega}(g)a^r. 
$$
Hence, our notion of ``variation" $v_{j,\om}(\cdot)$ is stronger than the one in \cite{MarShif1} with $\delta=a$. Indeed, we also require $\om$-continuity in each coordinate. Therefore, all the Berry-Esseen theorems in \cite{MarShif1} hold true in our setup. 
 \end{remark}
Next, let $f_j:\cX_{j,\infty}\to\bbR$ be measurable functions such that $\sup_j\|f_j\|_{j,\om}<\infty$. Denote 
$$
S_nf=\sum_{j=0}^{n-1}f_j(X_j,X_{j+1},...).
$$
In this paper we will prove local central limit theorems for the partial sums $S_nf$.

\subsection{On the behavior of the variance and the CLT}

Using Theorem \ref{RPF} and the argument in the proof of \cite[Theorem 2.16]{MarShif1} we get the following result.

\begin{theorem}\label{Var them}
Under Assumptions \ref{MixAss} and \ref{Assss} the following conditions are equivalent.
\vskip0.2cm
(1) $\liminf_{n\to\infty}\text{Var}(S_nf)<\infty$;
\vskip0.2cm
(2) $\sup_{n\in\bbN}\text{Var}(S_n f)<\infty$;

\vskip0.2cm
(3) we can write 
$$
f_j=\bbE[f_j(X_j,X_{j+1},...)]+M_j(X_j,X_{j+1},...)+u_{j+1}(X_{j+1},X_{j+2},...)-u_j(X_j,X_{j+1}...),\,\ \text{a.s.}
$$
where $\sup_j\|u_j\|_{j,\om}<\infty,$\;
$\sup_j\|M_j\|_{j,\om}\!\!<\!\!\infty$, \, $u_j$ and $M_j$ have zero mean and $M_j(X_j,X_{j+1},...),j\geq 0$ is a reverse martingale difference  with respect to the reverse filtration $\cG_j=\sigma\{X_k: k\geq j\}$ and\footnote{Note that by the martingale convergence theorem we get that the sum $\sum_{k=0}^\infty M_k(X_k,X_{k+1},...)$ converges almost surely and in $L^p$ for all finite $p$.} 
$$
\sum_{j\geq 0}\text{Var}(M_j(X_j,X_{j+1},...))\!\!<\!\!\infty.
$$
\vskip0.1cm
(4) there exist measurable functions $H_j:\cX_{j,\infty}\to\bbR$ such that 
$$
f_j(X_j,X_{j+1},...)=H_{j+1}(X_{j+1},X_{j+2},...)-H_j(X_j,X_{k+1},...)\,\,\text{ a.s.}
$$
In fact, we must have $H_j\in L^s(\ka_j)$ for all finite $s\leq b$ and, in fact, $H_j=\mu_j(H_j)+u_j(X_j,X_{j+1},...)+\sum_{k\geq j}M_k(X_k,X_{k+1},...)$.
\end{theorem}

\begin{remark}
In \cite{MarShif1} we showed that if $\lim_{n\to\infty}\text{Var}(S_nf)=\infty$ then $(S_nf-\bbE[S_nf])/\sigma_n$ converges in distribution to the standard normal law, where $\sig_n=\|S_nf-\bbE[S_nf]\|_{L^2}$, namely the CLT holds. In fact, we get optimal CLT rates.
\end{remark}
\subsection{Main results: reduction theory, local limit theorems and first order Edgeworth expansions}\label{Sec LLT}

Recall that a sequence of  square integrable random variables $W_n$ with $\sig_n=\|W_n-\bbE[S_n]\|_{L^2}\to\infty$ obeys the non-lattice local central limit theorem (LLT) if for every continuous function $g:\bbR\to\bbR$ with compact support, or an indicator of a bounded interval we have 
$$
\sup_{u\in\bbR}\left|\sqrt{2\pi}{\sig_n}
\bbE[g(W_n-u)]-\left(\int g(x)dx\right)e^{-\frac{(u-\bbE[W_n])^2}{2\sig_n^2}}\right|=
 o(1).
$$

\noindent
A sequence of square integrable integer valued  random variables $W_n$  
obeys the lattice LLT if 
$$
\sup_{u\in\bbZ}\left|\sqrt{2\pi}{\sig_n}\bbP(W_n=u)
-\eta e^{-\frac{(u-\bbE[W_n])^2}{2\sig_n^2}}\right|= o(1).
$$
 To compare the two results note that the above equation is equivalent to saying 
that
for every continuous function $g:\bbR\to\bbR$ with compact support, 
$$
\sup_{u\in\bbZ}\left|\sqrt{2\pi}{\sig_n}\bbE[g(W_n-u)]
-\left(\sum_{k}g(k )\right)e^{-\frac{(u-\bbE[W_n])^2}{2\sig_n^2}}\right|= o(1).
$$

 A more general type of LLT which is valid also for reducible $(f_j)$ (defined below)
 is discussed in Section \ref{Red sec}.

\begin{definition}\label{Irr Def}
A sequence of real valued functions $(f_j)$ on $\cX_{j,\infty}$ is reducible (to a lattice valued 
sequence) 
if there is $h\not=0$ and 
 functions $H_j: \cX_{j,\infty}\to \mathbb{R},$ $Z_j: \cX_{j,\infty}\to \mathbb{Z}$  
such that $\sup_j
\|H_j\|_{j,\om}<\infty$, $ \left(S_nH\right)_{n=1}^\infty$ is tight, and
\begin{equation}
\label{EqReduction}
f_j=\ka_j(f_j)+H_j+hZ_j.
\end{equation}

A sequence of  $\mathbb{Z}$ valued functions $(f_j)$ is reducible if it admits the
representation \eqref{EqReduction} as above with
 $h>1.$
\vskip1mm

We say that the sequence $(f_j)$ is irreducible if it is not reducible.
\end{definition}

\begin{remark}\label{Rem Red}
By Theorem \ref{Var them}
the tightness condition in the definition of reducibility means that  
$$
H_j(x_j,x_{j+1},..)=A_j(x_j,x_{j+1},..)+B_j(x_j,x_{j+1},...)-B_{j+1}(x_{j+1},x_{j+2},...)
$$ 
where $\sup_j\|A_j\|_{j,\om}<\infty$ and $\sup_j\|B_j\|_{j,\om}<\infty$ and $\sum_{j\geq 0}A_j(X_j,X_{j+1},...)$ converges almost surely and in $L^p$ for all $p$.
Indeed, if the variance of $S_nH$ diverges then $S_nH/\|S_nH\|_{L^2}$ is not tight since it satisfies the central limit theorem (see \cite{MarShif1}). 
\end{remark}

Next, let $R=R(f)$ be the set of all numbers  $h\not=0$ such that \eqref{EqReduction} holds with appropriate functions $H_j, Z_j$. If $(f_j)$ is 
irreducible then $R=\emptyset$. Moreover, by Theorem \ref{Var them}, if $\sigma_n\not\to\infty,\, \sig_n=\|S_nf-\bbE[S_nf]\|_{L^2}$ then
$R=\bbR\setminus\{0\}$ since then \eqref{EqReduction} holds  for any $h$ with $Z_j=0$. The following result completes the picture. Let $\mathsf{H}=\{1/r: r\in R\}\cup \{0\}$.

\begin{theorem}\label{Reduce thm}
Let Assumptions \ref{MixAss}, \ref{Assss}, \ref{Ball Ass} be in force. 
 If $(f_j)$ is reducible and $\sigma_n\to \infty$   then 
$$
\mathsf{H}=h_0\bbZ
$$
for some $h_0>0$. As a consequence, the number $r_0=1/h_0$ is the largest positive number such that $(f_j)$ is reducible to an $r_0\bbZ$-valued sequence. Therefore, $f_j$ can be written in the form \eqref{EqReduction} with an irreducible sequence $(Z_j)$. 
\end{theorem}

Now we are ready to fomulate our results concerning the LLT.
\begin{theorem}\label{LLT1}
Let Assumptions \ref{MixAss}, \ref{Assss}, \ref{Ball Ass} be in force.
Let $(f_j)$ be an irreducible sequence of $\mathbb{R}$ valued functions.
 Then the sequence of random variables
  $W_n=S_nf$ obeys the non-lattice LLT if  $\sig_n\to\infty$.    
\end{theorem}

As a byproduct of the arguments of the proof of Theorem \ref{LLT1} we also 
obtain the first order Edgeworth expansions.
\begin{theorem}\label{ThEdge1}
Let Assumptions \ref{MixAss}, \ref{Assss}, \ref{Ball Ass} be in force.
 If $(f_j)$ is irreducible then  with  $\sigma_n=\|S_nf-\bbE[S_nf]\|_{L^2}$ and $S_n=S_nf$,
   \begin{equation}
   \label{EdgeFirst}
\sup_{t\in\bbR}\left|\bbP((S_n-\bbE[S_n])/\sig_n\leq t)-\Phi(t)-
  \frac{\bbE[(S_n-\bbE[S_n])^3](t^3-3t)}{6\sig_n^3}\varphi(t)\right|=o(\sig_n^{-1})
   \end{equation}
   where
   $\varphi(t)=\frac{1}{\sqrt{2\pi}}e^{-t^2/2}$
   is the standard Gaussian density and 
   $ \Phi(t)=\int_{-\infty}^t \varphi(s) ds$ is the Gaussian cumulative distribution function.
\end{theorem}

We note that by  the calculations at the end of \cite[Section 6.6.1]{MarShif1} we have
$ \frac{\bbE((S_nf-\bbE[S_nf])^3)}{\sigma_n^3}=O(\sig_n^{-1})$ and so the correction term $ \frac{\bbE[(S_n-\bbE[S_n])^3](t^3-3t)}{6\sig_n^3}\varphi(t)$ is of order $O(\sig_n^{-1})$.
\\

 Theorem \ref{LLT1} leads naturally to the questions how to check irreducibility.
We obtain several results in this direction, extending the results obtained in
\cite{DolgHaf LLT} to corresponding results for Markov shifts.

\begin{theorem}\label{LLT2}
Let Assumptions \ref{MixAss}, \ref{Assss}, \ref{Ball Ass} be in force.
If the spaces $\cX_j$ are connected and $\sig_n\to\infty$ then $(f_j)$ is irreducible. Therefore  
$W_n=S_nf$ obeys the non-lattice LLT. 
\end{theorem}

\begin{theorem}\label{Thm small}
Let Assumptions \ref{MixAss}, \ref{Assss}, \ref{Ball Ass} be in force.
If $\|f_n\|_{L^\infty}\to 0$, $ \sup_n\|f_n\|_{j,\om}<\infty$  and $\sig_n\to\infty$ then $(f_n)$ is irreducible and so the non-lattice LLT holds.
\end{theorem}

Next, we consider the lattice case.
\begin{theorem}\label{LLT Latt}
Let Assumptions \ref{MixAss}, \ref{Assss}, \ref{Ball Ass} be in force.
Let $(f_j)$ be an irreducible sequence  of $\mathbb{Z}$ valued functions.
Then the sequence of random variables
  $W_n=S_nf$ obeys the lattice LLT 
   if $\sig_n\to\infty$.    
\end{theorem}

In fact, a  generalized lattice LLT holds for general sequences of reducible functions (see Theorem \ref{LLT RED}). This result includes Theorem \ref{LLT Latt} as a particular case, and  together with Theorem \ref{LLT1} we get a complete  description  of the local distribution of partial  sums for the Markov shifts considered in this manuscript. Since the formulation of Theorem \ref{LLT RED} is more complicated it is postponed to Section \ref{Red sec}.

\subsection{An overview: obstructions to LLT}
\label{SSObstructions}
 The results presented above show that  there are only three obstructions for local distribution of ergodic sums to be Lebesgue:
 
 (a) {\em lattice obstruction:} the individual terms could be lattice valued in which case the sums take values in the same lattice;
 
 (b) {\em summability obstruction:} the sums can converge almost surely in which case individual terms are too large 
 to ensure the universal behavior of the sum;

(c) {\em gradient (coboundary) obstruction:} if the observable  is of the form $f_j\!\!=\!\!B_j\!\!-\!\!B_{j+1}\circ T_j$ then
the sum telescopes and the variance does not grow. 

Obstruction (a) already appears in the the iid case, say for integer valued functions. Obstruction (b) is a general one, if the steps of the ``random walk" are of the same magnitude of the sum then the CLT fails in general. Obstructions (b) and (c) mean that the sequence $S_nf$ is tight, which also implies that the CLT fails. The combination of (b) and (c) is related to  the martingale coboundary
decomposition developed by Gordin \cite{Gor}. The only works in literature where all three obstructions appear are in \cite{DolgHaf LLT, DS}, in the context of the local CLT. In this paper we get the same picture, but for more general functionals of a wide class of Markov  shifts beyond the finite state ones considered in \cite{DolgHaf LLT}.
We refer the readers to \cite[Section 1.5]{DolgHaf LLT} a detailed discussion, historic perspective and references about these obstructions and the role they play in literature.


\section{Some immediate consequences of our standing assumptions}\label{Examples}
Let us define
$$
\Delta_{j,n,\om}=\sup_{G: |G|\leq 1}\,\,\sup_{x,y\in\cX_{j+n}, x\not=y}\,\left(\om(d_{j+n}(x,y))\right)^{-1}\left|\int G(z)P_{j,n}(dz|x)-\int G(z)P_{j,n}(dz|y)\right| 
$$
where $P_{j,n}(\cdot|x)$ is the law of $(X_j,X_{j+1},...,X_{j+n-1})$ given $X_{j+n}=x$.

\begin{proposition}\label{ContAss} 
Let Assumption \ref{Assss} be in force. Then for every $j$ and $n\geq n_0$, 
$$
\Delta_{j,n,\om}\leq A_1.
$$
\end{proposition}
\begin{proof}
Denote $P_{j}(dy|x)=P_{j,1}(dy|x)$. For every measurable function $G$ such that $|G|\leq 1$ we have
$$
\left|\int G(z)P_{j,n}(dz|x)-\int G(z)P_{j,n}(dz|y)\right|
$$
$$
=\Bigg|\int\dots\int G(z_n,...,z_1)(P_{j+n-1,n_0}(d(z_n,...,z_{n-n_0+1})|x)-P_{j+n-1,n_0}(d(z_n,...,z_{n-n_0+1})|y))
$$
$$
\times P_{j+n-1-n_0}(dz_{n-n_0}|z_{n-n_0+1})\cdots P_{j}(z_2,dz_1)\Bigg|
$$
$$
\leq\int \tilde G(z_n,...,z_{n-n_0+1})|(P_{j+n-1,n_0}(d(z_n,...,z_{n-n_0+1})|x)-P_{j+n-1,n_0}(d(z_n,...,z_{n-n_0+1})|y))|
$$
for some function $\tilde G$ which takes values in $[0,1]$. Thus, 
$$
\left|\int G(z)P_{j,n}(dz|x)-\int G(z)P_{j,n}(dz|y)\right|\leq \|P_{j+n-1,n_0}(\cdot|x)-P_{j+n-1,n_0}(\cdot|y)\|_{TV}\leq A_1\om(d_{j+n}(x,y)).
$$
\end{proof}

Next, recall that $\kappa_j$ denotes the law of $(X_j,X_{j+1},...)$ and that $\cX_{j,\infty}=\prod_{k\geq j}\cX_k$.
Let us define for $x\in\cX_{j,\infty}$ and $r>0$,
$$
\tilde B_{j}(x,r)=\{y\in\cX_{j,\infty}: \om_j(x,y)\leq r\}.
$$
\begin{proposition}\label{Ball ass2}
Under Assumption \ref{Ball Ass} we have:
\vskip0.1cm
(i) For all $r$ small enough we have
$$
\eta(r):=\inf_{j,x}\kappa_j(\tilde B_{j}(x,r))>0.
$$
\vskip0.1cm
(ii) For every $(b_j)\in\cX_{0,\infty}$ 
for $j,\ell$, $\te>0$ and $x\in\cX_{j+\ell-1,\infty}$ set
    $$
\Gamma_{j,\ell,\te,x}=\left\{(u,v): \max_{k<\ell}(a^{k}\om(d_{j+k}(u_{j+k},b_{j+k})))\leq\te/2, \,\,\max_{k<\ell}(a^{k}\om(d_{j+k}(v_{j+k},\tilde b_{j+k})) \leq \te/2\right\}\subseteq\left(\prod_{j\leq k<j+\ell}\cX_k\right)^2
$$
where $\tilde b_{j+k}=\tilde b_{j+k}(x)=b_{j+k}$ for $k<\ell-1$ and $\tilde b_{j+\ell-1}(x)=x_{j+\ell-1}$.
Then for $\te$  small enough and $\ell$ large enough
$$
\inf_{j,x}(P_{j+\ell}(\cdot|x)\times P_{j+\ell}(\cdot|x))(\Gamma_{j,\ell,\te,x})>0.
$$  
\end{proposition}

\begin{proof}
(i) For $x\in\cX_{j,\infty}$ and $r>0$ we have
$$
\kappa_j(\tilde B_j(x,r))=\bbP\left(\sup_{n\geq0}(a^{n}\om(d_{j+n}(X_{j+n},x_{j+n})))\leq (2A_1)^{-1}r\right).
$$
Now, since $\om(\cdot)\leq 1$ we see that if $n_r\geq n_0$ satisfies $(2A_1)a^{n_r}<r$ and $r$ is small enough
$$
\kappa_j(\tilde B_j(x,r))=\bbP\left((2A_1)\om(d_{j+n}(X_{j+n},x_{j+n}))\leq a^{-n}r: \forall n<n_r\right)\geq \left(\zeta\left((2A_1)^{-1}r\right)\right)^{[n_r/n_0]+1}:=\eta(r)>0 
$$
where $\zeta(r)>0$ comes from Assumption \ref{Ball Ass}. The verification of part (ii) is similar since the set $\Gamma_{j,\ell,\te,x}$ is just a product of two ``balls" $\tilde B_{j,n}$ in $\prod_{j\leq k<j+\ell}\cX_k$.
\end{proof}

\section{Non-autonomous functional analytic tools}
\subsection{Norms of functions}
\begin{lemma}\label{al be}
Let us take $r\in(0,1)$.
If $g_j:\cX_{j,\infty}\to\bbR$ satisfy $\sup_{j}\|g\|_{j,\om}<\infty$ and $\lim_{j\to\infty}\|g\|_{L^1(\ka_j)}=0$ then $\lim_{j\to\infty}\|g\|_{j,\omega^{r}}=0$.   
\end{lemma}
\begin{proof}
First, we have 
$$
|g_j(x)-g_j(y)|\leq 2\|g\|_\infty^{1-r}|g_j(x)-g_j(y)|^{r}\leq C\|g_j\|_{\infty}^{1-r}(\omega_j(x,y))^r
$$
and so it is enough to show that $\|g_j\|_{\infty}\to 0$ as $j\to\infty$. To prove that, we note that there is a constant $A>0$ such that
$$
\left|g_j(x)-\frac{1}{\kappa_j(\tilde B_j(x,\varepsilon))}\int_{\tilde B_j(x,\varepsilon)}g_j(y)d\ka_j(y)\right|\leq A\sup_{u\leq \varepsilon}\omega(u),
$$
where we recall that $\tilde B_j(x,\varepsilon)=\{y: \om_j(x,y)\leq\varepsilon\}$ (note that by Proposition \ref{Ball ass2} this set has positive $\kappa_j$ measure). Thus, by taking $j\to\infty$ we see that 
$$
\limsup_{j\to\infty}\|g_j\|_{\infty}\leq A\sup_{u\leq \varepsilon}\omega(u).
$$
Since this is true for every $\varepsilon$  and $\lim_{u\to 0}\om(u)=0$ we conclude that $\lim_{j\to\infty}\|g_j\|_{\infty}=0$.
\end{proof}
\subsection{A sequential Perron Frobenius theorem}
Define
$$
\cL_j g(x)=\bbE[g(X_{j},X_{j+1},...)|(X_{j+k})_{k\geq 1}=x]=\int g(b,x)P_{j,1}(db,x).
$$
Denote 
$$
\cL_j^n=\cL_{j+n-1}\circ\cdots\circ\cL_{j+1}\circ\cL_j.
$$
Then 
$$
\cL_j^n g(x)=\bbE[g(X_{j},X_{j+1},...)|(X_{j+k})_{k\geq n}=x]=\int g(b,x)P_{j,n}(db,x).
$$

\begin{theorem}\label{RPF}
Let Assumptions \ref{MixAss}, \ref{Assss} and \ref{Ball Ass} be in force.
For every  $j\geq0$, $n\in\bbN$ and $g:\cX_{j,\infty}\to\bbR$ with $\|g\|_{j,\om}<\infty$,
$$
\|\mathcal L_j^n(g)-\kappa_j(g)\|_{j+n,\omega}\leq C\|g\|_{j,\omega}\zeta^n
$$
for some constants $C>0$ and $\zeta\in(0,1)$ which depend (explicitly) only on $n_0$, $\omega$, $\delta$ and $\gamma$ from Assumptions \ref{MixAss}, \ref{Assss} and \ref{Ball Ass}.
\end{theorem}
\begin{proof}
Let us take a function $g$ such that $\|g\|_{j,\omega}\leq1$.
Let us denote 
$$
\bar g_k(x_j,...,x_{j+k-1})=\inf_{y\in\cX_{j+k,\infty}}g(x_j,...,x_{j+k-1},y).
$$
Then, since $\om(\cdot)\leq 1$,
$$
\sup_{x\in\cX_{j,\infty}}|g_k(x)-\bar g_k(x_j,...,x_{j+k-1})|\leq 2A_1 a^{k}.
$$
Therefore,    
$$
\|\cL_{j}^n g-\kappa_j(g)\|_{L^\infty}\leq  2A_1a^{[n/2]}+\|\bbE[\bar g_{[n/2]}|\cF_{j+n,\infty}]-\bbE[\bar g_{[n/2]}]\|_{L^\infty}.
$$
Next, recall that (see \cite[Ch.4]{Brad}),
$$
\phi_R(n)=\frac12\sup_k\sup\{\|\bbE[h|\cF_{k+n,\infty}]-\bbE[h]\|_{L^\infty}:h\in L^\infty(\cF_{0,k}): \|h\|_{L^\infty}\leq1\}.
$$
Using this together with Assumption \ref{MixAss} we see that
$$
\|\cL_{j}^n g-\kappa_j(g)\|_{L^\infty}\leq 4A_1a^{[n/2]}+2\phi_R([n/2])\leq 4A_1a^{[n/2]}+2A\gamma^n.
$$
Next, to estimate $v_{j+n,\om}(\cL_{j}^n g-\kappa_j(g))$ it is clearly enough to estimate it when $\kappa_j(g)=0$. We have 
$$
\cL_{j}^n g(x)-\cL_{j}^n g(y)=\int g(b,x)P_{j,n}(db|x)-\int g(b,y)P_{j,n}(da|y)
$$
$$
=\int (g(b,x)-g(b,y))P_{j,n}(db|x)+\int g(b,y)(P_{j,n}(db|x)-P_{j,n}(db|y)):=I_1+I_2.
$$
Now, as 
$$
|g(b,x)-g(b,y)|\leq\om_j((b,x),(b,y))\leq a^{n}\om_{j+n}(x,y)
$$
we see that 
$$
|I_1|\leq a^{n}\om_{j+n}(x,y)
$$
On the other hand by  and Proposition \ref{ContAss},
$$
|I_2|\leq A_1\om(d_{j+n}(x,y))=\frac12(2A_1\om_{j+n}(x,y))\leq\frac12\om_{j+n}(x,y). 
$$
Thus, 
$$
v_{j+n,\om}(\cL_{j}^ng)\leq \frac12+a^{n}. 
$$
Taking $n$ large enough we see that 
$$
\|\mathcal L_j^n(g)-\kappa_j(g)\|_{j+n,\om}\leq\frac34\|g\|_{j,\om}.
$$
Thus, by iterating we see that 
$$
\|\mathcal L_j^n(g)-\kappa_j(g)\|_{j+n,\om}\leq C'\|g\|_{j,\om}\zeta^n
$$
for some $C'>0$ and $\zeta\in(0,1)$.
\end{proof}

\subsection{Norms of complex transfer operators}
\subsubsection{\textbf{A complex Perron-Frobenius theorem}}
For a complex number $z$ let us define 
$$
L_{j,z}(h)=\cL_j(e^{zf_j}h).
$$
Let us also define $\cL_{j,t}=L_{j,it}$.
Denote by $B_j=\cB_{j,\om}$  the space of all functions $g:\cX_{j,\infty}\to\bbC$ such that $\|g\|_{j,\om}<\infty$. Let $B_j^*$ denote its dual space.
Then, since $\sup_j\|f_j\|_{j,\om}<\infty$ we see that $L_{j,z}$ are uniformly analytic in $z$.
By applying \cite[Theorem D.2]{DolgHaf PTRF 2} we get the following corollary of Theorem \ref{RPF}. 
\begin{corollary}\label{Cor1}
Let Assumptions \ref{MixAss}, \ref{Assss} and \ref{Ball Ass} be in force.
There exists $0<\delta_0<1$ such that for every $z\in\bbC$ with $|z|\leq\delta_0$ there are $\lambda_j(z)\in\bbC\setminus\{0\}$, $h_j^{(z)}\in B_j$ and $\ka_j^{(z)}\in B_j^*$ such that $\ka_j^{(z)}(\textbf{1})=\ka_j^{(z)}(h_j^{(z)})=1$, 
 $\lambda_j(0)=1$, $h_j^{(0)}=\textbf{1}$, $\ka_j^{(0)}=\ka_j$
and
\begin{equation}\label{UP11}
 L_{j,z} h_j^{(z)}=\lambda_j(z)h_{j+1}^{(z)},\, (\mathcal L_{j,z})^*\ka_{j+1}^{(z)}=\lambda_j(z)\ka_{j}^{(z)}.   
\end{equation}
Moreover, $z\to\lambda_j(z)$, $z\to h_j^{(z)}$ and $z\to\ka_j^{(z)}$ are analytic functions of $z$ with uniformly (over $z$ and $j$) bounded norms.
Finally,  there are $C_1>0, \delta_1\in(0,1)$ such that for every $g\in B_j$ and all $n$,
\begin{equation}\label{Exp11}
 \left\|L_j^{z,n}g-\lambda_{j,n}(z)\ka_j^{(z)}(g)h_{j+n}^{(z)}\right\|_{j+n,\om}\leq C_1\|g\|_{j,\om}\delta_1^n
\end{equation}
where $\lambda_{j,n}(z)=\prod_{k=j}^{j+n-1}\lambda_k(z)$. 
\end{corollary}

\subsubsection{\textbf{Lasota Yorke inequality and related results}}
Let us next prove a Lasota York type inequality.
\begin{lemma}\label{Lasota Yorke}
Let Assumptions \ref{Assss} and \ref{Ball Ass} be in force.
 Let $T>1$. Then there exists $C_1=c(T)>0$ such that for every $h\in\cB_{j,\om}$, 
 $$
\sup_{|t|\leq T}\|\cL_{j,t}^n h\|_{j+n,\om}\leq C_1(\|h\|_{L^\infty}+a^{n}v_{j,\omega}(h)).
 $$
\end{lemma}
\begin{proof}
 First, without loss of generality, since the spaces are metric we can replace the essential supremum by the regular supremum. Indeed, the set of continuous functions  is separable and we can just disregard sets of measure $0$. Next, we have
 $$
\|\cL_{j,t}^n h\|_{L^\infty}\leq\|h\|_{L^\infty}.
 $$
Now, write
$$
\cL_{j,t}^n h(x)=\int e^{itS_{j,n}f(b,x)}h(b,x)P_{j,n}(db|x).
$$
Then
$$
|\cL_{j,t}^n h(x)-\cL_{j,t}^n h(y)|\leq I_1+I_2+I_3 
$$
where 
$$
I_1=|t|\int |S_{j,n}f(b,x)-S_{j,n}f(b,y)||h(b,x)|P_{j,n}(db|x), \,\,\,I_2= \int |h(b,x)-h(b,y)|P_{j,n}(db|x)
$$
and 
$$
I_3=\int |h(b,y)||P_{j,n}(db|x)-P_{j,n}(db|y)|.
$$
Notice that 
$$
|S_{j,n}f(b,x)-S_{j,n}f(b,y)|\leq \sum_{k=j}^{j+m-1}|f_k(b_k,...,b_{j+n-1},x)-f_k(b_k,...,b_{j+n-1},y)|
$$
$$
\leq C\sum_{k=j}^{j+n-1}a^{(j+n-k)}\om_{j+n}(x,y)\leq C\om_{j+n}(x,y)
$$
for some constants $C$ that depend only on $\sup_kv_{k,\om}(f_k)$. 
Thus,
$$
I_1\leq CT\|h\|_{L^\infty}\om_{j+n}(x,y).
$$
Moreover,
$$
|h(b,x)-h(b,y)|\leq v_{j,\om}(h)\om_j((b,x),(b,y))\leq v_{j,\om}(h)\om_j((b,x),(b,y))\leq v_{j,\om}(h)a^{n}\om_{j+n}(x,y).
$$
Therefore, 
$$
I_2\leq v_{j,\om}(h)a^{n}\om_{j+n}(x,y)
$$
Finally, by Proposition \ref{ContAss}, 
$$
I_3\leq A_1\|h\|_{L^\infty}\om_{j+n}(x,y).
$$
\end{proof}

 Next, let us define   
 $$
\|g\|_{j,*}=\max(\|g\|_{L^\infty},v_{j,\om}(g)/(2C_1)).
 $$
Then by Lemma \ref{Lasota Yorke} and taking into account that $\|\cL_{j,t}^nh\|_{L^\infty}\leq \|h\|_{L^\infty}$ we have get the followin result.
\begin{corollary}\label{Corollary 4.8}
Let $k_0$ be the smallest positive integer such that $2C_1 a^{k_0}\leq 1$. Then
\begin{equation}\label{Norm bound}
\sup_{|t|\leq T}\sup_j\sup_{n\geq k_0}\|\cL_{j,t}^n\|_{j+n,*}\leq 1.   
\end{equation}
\end{corollary}

Next we need:
\begin{lemma}\label{Lemma 4.6}
Let Assumptions \ref{Assss} and \ref{Ball Ass} be in force.
For every $\varepsilon\in(0,1)$ there exists $k_1(\varepsilon)$   such that for every function $H_j:\cX_{j,\infty}\to\bbR$ such that  $\|H\|_{j,*}\leq 1$ and 
 $|\kappa_j(H)|\leq1-\varepsilon$  for all $k\geq k_1(\varepsilon)$ we have $\|\cL_{j}^kH\|_{j+k,*}\leq 1-\varepsilon/2$.
\end{lemma}
\begin{proof}
Let $C>1$ and $\zeta\in(0,1)$ be like in Theorem \ref{RPF}.
By Theorem \ref{RPF} we have
$$
\|\cL_j^kh-\kappa_j(h)\|_{j+k,\om}\leq C\|h\|_{j,\om}\zeta^k\leq \zeta^{k/2}
$$
if $C\zeta^{k/2}\leq1$.
Thus 
$$
\|\cL_j^kh\|_{j+k,*}\leq 1-\varepsilon+ \zeta^{k/2}<1-\varepsilon/2
$$
if $\zeta^{k/2}\leq \varepsilon/2$.
\end{proof}
We will also need the following result.
\begin{lemma}\label{Le2}
Let Assumption \ref{Ball Ass} be in force and suppose that $\|h\|_*\leq1$. The for every $\beta\in(0,1)$ small enough there exists a constant $u(\beta)>0$ which does not depend on $h$ such that $\inf|h|<1-\beta$ implies that  $\bbE[|h|]<1-u(\beta)$. Moreover, we can ensure that $u(\beta)$ is decreasing in $\beta$ and $\lim_{\beta\to 0}u(\beta)=0$.
\end{lemma}
\begin{proof}
Suppose that
$$
|h(x)|<1-\beta
$$
for some $x$ and $\beta>0$. Then for every $r$,
$$
\mu_j(|h|)\leq 1-\mu_j(\tilde B_j(x,r))+\mu_j(\tilde B_j(x,r))\sup\{|h(y)|: y\in B_j(x,r)\}.
$$
Taking $r$ small enough we get
that the above supremum does not exceed $1-\beta/2$. Therefore, by Proposition \ref{Ball ass2}
$$
\mu_j(|h|)\leq 1-\beta\eta(r)/2
$$
where $\eta(r)$ is the lower bound from  Proposition \ref{Ball ass2}.
\end{proof}

The following consequence of the above results will be crucial in our proofs.
\begin{corollary}\label{Corollary 4.9}
Under Assumption \ref{Ball Ass} we have the following.
Let $\varepsilon\in(0,1/2)$. Then the exist $k_2=k_2(\ve)$ and $\te(\varepsilon)>0$
with the following property.
If $\|\cL_{j,t}^{k_2+m}\|_{j+k_2+m,*}> 1-\varepsilon/2$  for some $m\geq k_0$ and $t\in[-T,T]$ then there exists $h$ with $\|h\|_{j,*}\leq 1$ and $\inf_x|\cL_{j,t}h(x)|\geq 1-u^{-1}(\varepsilon)$. Therefore, by reversing the roles of $\varepsilon$ and $u^{-1}(\varepsilon)$ there exists $\te(\varepsilon)>0$ such that $\|\cL_{j,t}^{k_2+m}\|_{j+k_2+m,*}> 1-\te(\varepsilon)$ implies $\inf_x|\cL_{j,t}h(x)|\geq 1-\varepsilon$ for some  $h$ with $\|h\|_{j,*}\leq 1$.
\end{corollary}
\begin{proof}
Let $k_2^*(\varepsilon)$ be the smallest positive  integer $k$ so that $C_1a^{k}\leq \frac12(1-\varepsilon/2)$. Then by Lemma \ref{Lasota Yorke} for every function $H$ such that $\|H\|_*\leq1$ and all $j\in\bbZ$, $s\geq  k_2^*(\varepsilon)$ and $t\in[-T,T]$ we have 
\begin{equation}\label{Upppp}
 \frac{v_{j+s,\om}(\cL_{j,t}^{s}H)}{2C_1}\leq 1-\varepsilon/2.
\end{equation}
In the course of the proof we will write $\|g\|_{*}$ instead  $\|g\|_{j,*}$ for all the functions $g$ and the indexes $j$ appearing below.
Next,  take $k_2(\varepsilon)\!\!=\!\!\max(k_2^*(\varepsilon), k_1(\varepsilon))$. 
Suppose that $\|\cL_{j,t}^{k_2(\varepsilon)+m}\|_{*}\!>\!1\!\!-\!\! \varepsilon/2$. Then  there is $h$ such that $\|h\|_*\!\leq\!\! 1$ and $\|\cL_{j,t}^{k_2(\varepsilon)+m}h\|_{j+k,*}
\!>\!1\!\!-\!\! \varepsilon/2$. Set $H=\cL_{j,t}^m h$. Then
$$
\|\cL_{j,t}^{k_2(\varepsilon)+m}h\|_{*}=\|\cL_{j+m,t}^{k_2(\varepsilon)}H\|_{*}=\max\left(\|\cL_{j+m,t}^{k_2(\varepsilon)}H\|_{L^\infty}, \frac{v_{j+m,\om}(\cL_{j+m,t}^{k_2(\varepsilon)}H)}{2C_1}\right)>1-\ve/2.
$$
Now, since $\|h\|_*\leq 1$ and $m\geq k_0$, it follows from \eqref{Norm bound} that $\|H\|_*\leq 1$. Thus, since $k_2(\varepsilon)\geq k_2^*(\varepsilon)$ we conclude from \eqref{Upppp} that 
$$
 \frac{v_{j+m,\om}(\cL_{j+m,t}^{k_2(\varepsilon)}H)}{2C_1}\leq 1-\varepsilon/2.
$$
Hence 
$$
\|\cL_{j+m,t}^{k_2(\varepsilon)}H\|_{L^\infty}>1-\ve/2,
$$
and so
$$
\|\cL_{j+m}^{k_2(\varepsilon)}|H|\|_{L^\infty}>1-\ve/2.
$$
Thus, since $k_2(\varepsilon)\geq k_1(\varepsilon)$ by (the contrapositive of)  
Lemma \ref{Lemma 4.6} we have
$$
\bbE[|H|]>1-\varepsilon
$$
and so by Lemma \ref{Le2} we conclude that 
$$
\inf |H|>1-u^{-1}(\varepsilon).
$$
\end{proof}

\section{Integral of characteristic function and LLT}
\label{SSInt-LLT}
Arguing
like in \cite[\S 2.2]{HK}, in order 
to prove a non-lattice LLT starting with a measure of the form $\ka_0=q_0d\mu_0$ with $\|q_0\|_\al<\infty$ it suffices to prove the following:
\vskip0.2cm
(i) there are constants $\delta,c_2,C_2>0$ such that   for all $|t|\leq \delta$  and all $n$
\begin{equation}\label{Small t}
\left\|\mathcal{L}_{0,t}^n\right\|_*\leq C_2 e^{-c_2\sig_n^2 t^2}.
\end{equation}
\vskip0.2cm

(ii) for each $T>\delta$  we have
\begin{equation}\label{Suff}
\int_{\del\leq |t|\leq T}\|\cL_{0, t}^n\|_{*}dt=o(\sig_n^{-1}).   
\end{equation}


\noindent
Similarly, in order 
to prove a lattice LLT for integer valued observables $f_j$ it suffices to  prove (i) and 
 \begin{equation}\label{SuffLat}
\mathrm{(ii)'} \quad \quad \int_{\del\leq t\leq 2\pi-\del}\|\cL_{0, t}^n\|_{*}dt=o(\sig_n^{-1}).   
\end{equation}


Henceforth we will assume that  Assumptions \ref{MixAss}, \ref{Assss} and \ref{Ball Ass} are in force without stating that directly.
We begin with \eqref{Small t}. The following result can be proven exactly like in \cite[Proposition 10]{DolgHaf LLT} (which is based on ideas in \cite{HK} and \cite{DolgHaf PTRF 1, DolgHaf PTRF 2}).

\begin{proposition}\label{PrSmVarBlock}
There are positive constants $\del_0,c_1,c_2,C_1,C_2$ such that for every finite sequence of functions    
$(v_j)_{j=n}^{n+m-1}$ with $a:=\max\|v_j\|_{j,\omega}\leq \del_0$ 
  we have the following. 
Set $\cA_{j}(g)=\cL_j(e^{iv_j}g)$ and 
$
\cA_{n}^m=\cA_{n+m-1}\circ \cdots\circ\cA_{n+1}\circ\cA_n.
$
Then
\begin{equation}
\label{AUpperLower}
C_1e^{-c_1\mathrm{Var}(S_{n,m} v)}\leq \|\cA_{n}^{m} \|_{\om} \leq C_2  e^{ -c_2 \mathrm{Var}(S_{n,m} v)}
\end{equation} 
where $S_{n,m}v=\sum_{k=0}^{m-1}v_{n+k}\circ T_n^k$ and $\|\cA_{n}^{m}\|_{\om}$ is the operator norm of $A_n^m$ with respect to the norm $\|\cdot\|_{n,\om}$ on the domain and $\|\cdot\|_{n+m,\om}$ on the range. 
\end{proposition}

\begin{corollary}
There exists $\del>0$ such that \eqref{Small t} holds for all $t\in[-\del,\del]$ and all $n$.    
\end{corollary}

\begin{proof}
We apply Proposition \ref{PrSmVarBlock}    with functions of the form $v_j\!\!=\!\!tf_j$. Let $\|f\|\!\!=\!\!\sup_j\|f_j\|_{j,\om}$. Now \eqref{Small t} follows from the upper bound in Proposition \ref{PrSmVarBlock} if $|t|\|f\|\leq \del_0$.  
\end{proof}

\subsection{Corange}
 Here we prove Theorem \ref{Reduce thm}. To simplify the proof we assume that $\mu_k(f_k)=0$ for all $k$.

Recall the definition of the set $\mathsf{H}$ in Theorem \ref{Reduce thm}. Let fix some $r\in(0,1)$.
Taking into account Remark \ref{Rem Red} 
and Lemma \ref{al be},  
$\mathsf{H}$ is the set of all real numbers $t$ such that for all $n$ we have 
\begin{equation}
\label{TimesTIntRed}
tf_n=h_n-h_{n+1}\circ T_n+g_n+Z_n
\end{equation}
where $h_n,g_n$ and $Z_n$ are functions such that $\sup_j\|h_j\|_{j,\om}\!\!<\!\!\infty$, 
$\mu_n(g_n)\!\!=\!\!0$, $ \sup_n\text{Var}(S_n g)\!\!<\!\!\infty$, 
$\|g_n\|_{\om^{r}}  \to 0
$
and $Z_n$ is integer valued.
 It is clear that $\mathsf{H}$ is a subgroup of $\bbR$.

 Theorem \ref{Reduce thm} follows from the following corollary of Proposition \ref{PrSmVarBlock}, whose proof, relying on Lemma \ref{al be}, proceeds exactly like in \cite[Corollary 4.13]{DolgHaf LLT}, but is given here for the sake of completeness.

\begin{corollary}\label{H corr}
\,

(i) If $\|S_nf\|_{L^2}\not\to \infty$ then $\mathsf{H}=\bbR$.
\vskip0.2cm

(ii) If $(f_j)$ is irreducible then $\mathsf{H}=\{0\}$.
\vskip0.2cm

(iii) If $(f_j)$ is reducible and $\|S_nf\|_{L^2}\to \infty$    then 
$
\mathsf{H}=t_0\bbZ
$
for some $t_0>0$. 

As a consequence, the number $h_0=1/t_0$ is the largest positive number such that $(f_j)$ is reducible to an $h_0\bbZ$-valued sequence. 
\end{corollary}
\begin{proof}
(i) If the variance does not diverge to $\infty$ then by Theorem \ref{Var them} we see that $(f_j)$ is decomposed as a sum of a center tight sequence and a coboundary. Thus for every real $t$ the function $tf_j$ has decomposition \eqref{TimesTIntRed} (with $Z_j=0)$ with $g_j$ being the martingale part. Since 
$
\sum_j\text{Var}(g_j)<\infty
$
we conclude that $\mathsf{H}=\bbR$.
\vskip0.2cm
(ii) If $t\not=0$ belongs to $\mathsf{H}$ then $f$ must be reducible to an $h\bbZ$-valued sequence, with $h=1/|t|$. 
\vskip0.2cm
(iii) 
We claim first that if $t\in \mathsf{H}$ then for all  $j$ large enough  the norms $\|\cL_{j,t}^n\|_{j,\om^r}$ do not converge to $0$ as $n\to\infty$.  Since $ \sup_{j\geq n}\|g_j\|_{j,\om^r}\to 0$ as $n\to\infty$, for $j$ large enough  we can apply  Proposition \ref{PrSmVarBlock} with $\om^r$ instead of $\om$ and with these functions to conclude 
that 
$$
\|\cA_{j,t}^n\|_{\om^r}\geq C_1e^{-c_2\text{Var}(S_{j,n}g)}
$$
where $\cA_{j,t}^n$ is the operator given by
$$
\cA_{j,t}^n(q)=\cL_j^n(e^{it S_{j,n}g}q)
$$
and $\|\cdot\|_{\om^r}$ is its operator norm with respect to the norm $\|\cdot\|_{j,\om^r}$ on the domain and $\|\cdot\|_{j+n,\om^r}$ on the range.
Since $ \sup_n\text{Var}(S_{j,n}g)<\infty$ we get
$
\inf_{n}\|\cA_{j,t}^n\|_{\om^r}>0.
$
To finish the proof of the claim, note that 
$
\cA_{j,t}^n(q)=e^{ih_{j+n}}\cL_{j,t}^n(qe^{-ih_j})
$
and so
$
 \|\cA_{j,t}^n\|_{\om^r}\leq C\|\cL_{j,t}^n\|_{\om^r}
$
for some constant $C>0$.

Now, in order to complete the proof of the corollary, it is enough to show that there exists $\del_1>0$ such that for every $t\in \mathsf{H}$ and all $s\in[-\del_1,\del_1]\setminus\{0\}$ for all $j$ we have 
\begin{equation}\label{this will}
\lim_{n\to\infty}\|\cL_{j,t+s}^n\|_{\om^r}=0.
\end{equation}
\eqref{this will} shows that $t+s\not\in \mathsf{H}$. It  follows that $\mathsf{H}$ is a non empty discrete subgroup of $\mathbb{R}$,   whence $\mathsf{H}=t_0 \bbZ$
for some $t_0\in \bbR_+$,
completing the proof of the proposition.

In order to prove \eqref{this will}, let $ \|f\|=\sup_j\|f_j\|_{j,\om^r}$. Define $\del_1=\frac{\del_0}{2\|f\|+2}$, where $\del_0$ comes from Proposition \ref{PrSmVarBlock}. Thus, when $j$ is large enough and $|s|\leq \del_1$ we can  apply Proposition \ref{PrSmVarBlock} with the functions $u_j=g_j+sf_j$ to conclude that 
the operator 
$$
\cB_{j,t,s}^n(q)=\cL_j^n(qe^{it S_{j,n}g+isS_{j,n}f})
\quad\text{satisfies}\quad
\|\cB_{j,t,s}^n\|_{\om^r}\leq C_2'e^{-c_2s^2\text{Var}(S_{j,n}f)}
$$
where  we have used that the variance of $S_{j,n}g$ is bounded in $n$. The desired estimate
\eqref{this will} follows since
$
\lim_{n\to\infty}\text{Var}(S_{j,n}f)=\infty
$
and 
$\
\cL_{j,t+s}^n(q)=e^{-ih_{j+n}}\cB_{j,t,s}^n(qe^{ih_j}).
$
  \end{proof}

\section{A reduction lemma}\label{Sec 5}

\subsection{The apriori coboundary functions}

Next, let us take a designated point $\alpha=(\alpha_j)_{j\geq 0}\in\cX_{0,\infty}$. 
\begin{lemma}\label{Aux1}
Define $H_k:\cX_{k,\infty}\to\bbR$ by 
$$
H_k(x)=S_k f(\alpha_0,...,\alpha_{k-1},x)-S_kf(\alpha)=\sum_{m=0}^{k-1}\left(f_m(\alpha_m,...,\alpha_{k-1},x)-f_m(\alpha_m,\alpha_{m+1},...)\right).
$$
Then 
$\sup_k\|H_k\|_{k,\om}<\infty$
\end{lemma}
\begin{proof}
$$
\left|f_m(\alpha_m,...,\alpha_{k-1},x)-f_m(\alpha_m,\alpha_{m+1},...)\right|
$$
$$
\leq v_{m,\om}(f_m)\omega_m((\alpha_m,...,\alpha_{k-1},x),(\alpha_m,\alpha_{m+1},...))\leq a^{(m-k)}\om_{k}(x,(a_k,a_{k+1},...))
$$
and so 
$$
\sup_k\|H_k\|_{L^\infty}<\infty.
$$
Next, 
$$
|H_k(x)-H_k(y)|\leq \sum_{m=0}^{k-1}|f_m(\alpha_m,...,\alpha_{k-1},x)-f_m(\alpha_m,...,\alpha_{k-1},y)|\leq \sum_{m=0}^{k-1}v_{m,\om}(f_m)a^{(m-k)}\om_{k}(x,y)\leq C\om_k(x,y)
$$
for some constant $C>0$.
Thus $\sup_k\|H_k\|_{k,\om}<\infty$.
\end{proof}

Next, let us define $H_{j,\ell}:\cX_{j+\ell,\infty}\to\bbR$ by
$$
H_{j,\ell}(x)=S_{j,\ell}f(\alpha_j,...,\alpha_{j+\ell-1},x)-S_{j,\ell}f(\alpha_j,\alpha_{j+1},...).
$$
We also have the following result.
\begin{lemma}\label{Aux2}
Define $\Delta_{j,\ell}=H_{j,\ell+1}\circ T_{j+\ell}-H_{j,\ell}-(H_{j+\ell+1}\circ T_{j+\ell}-H_{j+\ell})$. Then
\begin{equation}\label{C1}
 \lim_{\ell\to\infty}\sup_{j}\|\Delta_{j,\ell}\|_{j+\ell,\omega}=0.   
\end{equation}
\end{lemma}
\begin{proof}
Let the functions $R_k$ and $R_{j,\ell}$
be given by
$$
R_{k}(X_k,X_{k+1},...)=S_{0,k}(\alpha_0,\alpha_1,...,\alpha_{k-1},X_k,X_{k+1},....).
$$
and 
$$
 R_{j,\ell}(X_{j+\ell},X_{j+\ell+1},...)=S_{j,\ell}(a_j,...,a_{j+\ell-1},X_{j+\ell},X_{j+\ell+1},...).
$$
Then
$$
H_{k}=R_{k}-S_{0,k}f(\alpha)\,\,\text{ and }\,\,H_{j,\ell}=R_{j,\ell}-S_{j,\ell}(\alpha_j,\alpha_{j+1},...)
$$
Now, denote by $T_j:\cX_{j,\infty}\to\cX_{j+1,\infty}$ the left shift and set $\al'_{j,k}(x)=(\alpha_j,...,\alpha_{j+k-1},x)$. Since
\begin{equation}\label{Cocycle}
S_{j,\ell}f(X_j,X_{j+1},...)=S_{0,j+\ell}f(X_0,X_1,...)-S_{0,j}f(X_0,X_1,...).
\end{equation}
we have
\begin{equation}\label{est1}
R_{j,\ell+1}\circ T_{j+\ell}-R_{j,\ell}=\left(S_{0,j+\ell+1}-S_{0,j}\right)\circ\al'_{j+\ell+1}\circ T_{j+\ell}-\left(S_{0,j+\ell}-S_{0,j}\right)\circ\al'_{j+\ell}
\end{equation}
$$
=R_{j+\ell+1}\circ T_{j+\ell}-R_{j+\ell}+\left(S_{0,j}f\circ\al'_{j+\ell}-S_{0,j}f\circ\al'_{j+\ell+1}\circ T_{j+\ell}\right).
$$
Since the points $\al'_{j+\ell}(x)$ and $\al'_{j+\ell+1}\circ T_{j+\ell}(x)$ have the same $j+\ell$ first coordinates and $\sup_k\|f_k\|_{k,\om}<\infty$ we see that 
$$
\sup_x|S_{0,j}f\circ\al'_{j+\ell}(x)-S_{0,j}f\circ\al'_{j+\ell+1}\circ T_{j+\ell}(x)|\leq Ca^\ell
$$
and for every $x,y$
$$
|S_{0,j}f\circ\al'_{j+\ell}(x)-S_{0,j}f\circ\al'_{j+\ell+1}\circ T_{j+\ell}(y)|\leq Ca^\ell\om_{j+\ell}(x,y).
$$
Therefore,
\begin{equation}\label{est2}
\|S_{0,j}f\circ\al'_{j+\ell}-S_{0,j}f\circ\al'_{j+\ell+1}\circ T_{j+\ell}\|_{j+\ell,\om}=O(a^\ell).
\end{equation}
Now, with $T_k^m=T_{k+m-1}\circ\cdots\circ T_{k+1}\circ T_k$ we have
\begin{equation}\label{H to R }
H_{j,\ell+1}\circ T_{j+\ell}-H_{j,\ell}=R_{j,\ell+1}\circ T_{j+\ell}-R_{j,\ell}+(S_{j,\ell}f\circ T_0^{j}(\alpha)-S_{j,\ell+1}\circ T_0^{j}(\alpha))
\end{equation}
$$
=R_{j,\ell+1}\circ T_{j+\ell}-R_{j,\ell}-f_{j+\ell}(T_0^{j+\ell}(\alpha))
\,\,\,\text{ and }
$$
$$
H_{j+\ell+1}\circ T_{j+\ell}-H_{j+\ell}=R_{j+\ell+1}\circ T_{j+\ell}-R_{j,\ell}+(S_{0,j+\ell}f(\alpha)-S_{0,j+\ell+1}f(\alpha))
$$
$$
=R_{j+\ell+1}\circ T_{j+\ell}-R_{j+\ell}-f_{j+\ell}(T_0^{j+\ell}(\alpha)).
$$  
Hence 
$$
\Delta_{j,\ell}=\left[R_{j,\ell+1}\circ T_{j+\ell}-R_{j,\ell}\right]-\left[R_{j+\ell+1}\circ T_{j+\ell}-R_{j+\ell}\right] 
$$
which together with \eqref{est1} and \eqref{est2} yields \eqref{C1}.
\end{proof}

Here we prove Theorem \ref{LLT2}. As we have explained before, it suffices to prove \eqref{Suff}. Hence Theorem \ref{LLT2} follows from the
estimate below.

\subsection{The reduction lemma}
\begin{lemma}
\label{LmTDSmall} 
For every $0<\delta<T$ and $L\in\bbN$ there exists $0<\delta(L)<1/2$ such that if
\begin{equation}
\label{L-Epsilon}
\varepsilon\leq \delta(L) 
\end{equation}
then  for every $j$ and $\ell\leq L+1$ we have the following.
If  for some nonzero real $t$ such that $\del \leq |t|\leq T$ there exists $h$ with 
$\|h\|_{\om,*}\leq 1$ and
\begin{equation}
\label{NormAlmostOne}
\inf_{x}|\cL_{j,t}^{\ell}h(x)|\geq 1-\varepsilon 
\end{equation}
then
\begin{equation}\label{Deccomp}
f_{j+\ell-1}=g_{j,t,\ell}+H_{j+\ell-1}-H_{j+\ell}\circ T_{j+\ell-1}+(2\pi/t) Z_{t,j,\ell}
\end{equation}
where $Z_{t,j,\ell}$ is integer valued, and for every $0<r<1$ we have 
$\sup_{t,j,\ell}\|g_{t,j,\ell}\|_{j+\ell,\om^r}\leq c_r(\ell)$ where $c_r(\ell)\to 0$ as $\ell\to\infty$.
\end{lemma}

\begin{remark}\label{Rem conn}
If the spaces $\mathcal X_j$ are connected we can take $Z_{t,j,\ell}$ to be a constant. Indeed, since all the functions $f_{j+\ell-1},g_{j,t,\ell},H_{j+\ell-1}, H_{j+\ell}\circ T_{j+\ell-1}$ are continuous we see that $Z_{t,j,\ell}$ is continuous and thus constant.
\end{remark}

\begin{remark}
The important part of the lemma is that the coboundary terms $H$ depend only on $j+\ell$, which will allow us to take the same coboundary parts 
 when \eqref{Deccomp} holds for both $j$ and $j+1$. The fact that $g$ and $Z$ are allowed to depend on $(j,\ell)$ does not cause any issues.
The appearance  of $H_{j+\ell}$ (as opposed to $H_{j,\ell}$) will yield the desired cancellation for  the sums $f_{j+\ell-1}+f_{j+\ell}\circ T_{j+\ell-1}$.  Such cancellations will be crucial for decomposing the summands inside
 such blocks into three components: coboudnaries, small terms  and a lattice valued variable.
The lattice valued variables will disappear after multiplication by $t$ and taking the exponents. The sum with small terms will be dealt with similarly to the case of small $t's$. 
\end{remark}

\begin{proof}[Proof of Lemma \ref{LmTDSmall}]
First, since $\inf_{x}|\cL_{j,t}^{\ell}h(x)|\geq 1-\ve$ then 
\begin{equation}\label{h bd}
\mu_j(|h|)\geq 1-\ve.
\end{equation} 
Now, by Lemma \ref{Le2},
exists $a(\ve)>0$ such that $\lim_{\ve\to 0}a(\varepsilon)=0$ and
$$
\inf|h|\geq 1-a(\varepsilon).
$$
Write $h(y)=r(y) e^{i \phi(y)}$, $r(y)=|h(y)|\geq 1-a(\ve)$.

Next, we have
$$ 
(\cL_{j,t}^\ell h)(x)=\int r(b,x)e^{it S_{j,\ell}f(b,x)+i\phi(b,x)}P_{j,\ell}(db|x). 
$$
Note that for any probability measure $\nu$ on a measurable space $M$ and  measurable functions $q:M\to \bbR$  and $r:M\to (0,\infty)$, 
$$
\left|\int_{M} r(u)e^{iq(u)}d\nu(u)\right|^2=\left(\int_{M} r(u)d\nu(u)\right)^2-2\int_{M}\int_{M}\sin^2\left(\frac12(q(u_1)-q(u_2))\right)r(u_1)r(u_2)\,d\nu(u_1)\,d\nu(u_2).
$$
Therefore,
$$
(1-\ve)^2\leq |\cL_{j,t}^\ell h(x)|^2
=\left(\int r(u,x)P_{j,\ell}(du|x)\right)^2
$$
$$
-2\int\int r(u,x)r(v,x)\sin^2\left(\frac12\left(t(S_{j,\ell}f(u,x)-S_{j,\ell}f(v,x))+\phi(u,x)-\phi(v,x)\right)\right)P_{j,\ell}(du|x)P_{j,\ell}(dv|x)
$$
$$
\leq 1-\int\int r(u,x)r(v,x)\sin^2\left(\frac12\left(t(S_{j,\ell}f(u,x)-S_{j,\ell}f(v,x))+\phi(u,x)-\phi(v,x)\right)\right)P_{j,\ell}(du|x)P_{j,\ell}(dv|x).
$$
Thus,
$$
2\int\int r(u,x)r(v,x)\sin^2\left(\frac12\left(t(S_{j,\ell}f(u,x)-S_{j,\ell}f(v,x))+\phi(u,x)-\phi(v,x)\right)\right)P_{j,\ell}(du|x)P_{j,\ell}(dv|x)\leq 2\ve+\ve^2\leq 3\ve. 
$$
Using that $r\geq 1-a(\ve)$ we see that 
$$
\int\int\sin^2\left(\frac12\left(t(S_{j,\ell}f(u,x)-S_{j,\ell}f(v,x))+\phi(u,x)-\phi(v,x)\right)\right)P_{j,\ell}(du|x)P_{j,\ell}(dv|x)\leq (3/2)\ve(1-a(\ve))^{-2}. 
$$
 If $\varepsilon$ is small enough we get that for all $x\in\cX_{j+\ell,\infty}$,
\begin{equation}\label{upp}
 \int\int\sin^2\left(\frac12\left(t(S_{j,\ell}f(u,x)-S_{j,\ell}f(v,x))+\phi(u,x)-\phi(v,x)\right)\right)P_{j,\ell}(du|x)P_{j,\ell}(dv|x)\leq 2\ve.   
\end{equation}

Next, fix some $x\in\cX_{j+\ell-1,\infty}$ and let $\te\in(0,1)$.
Define $\tilde\alpha=\tilde \alpha(x)=(\alpha_j,...,\alpha_{j+\ell-2},x)$ and
$$
\Gamma_{j,\ell,\te,x}=\left\{(u,v): \max_{k<\ell}(a^{k}\om(d_{j+k}(u_{j+k},\alpha_{j+k})))\leq\te/2, \,\,\max_{k<\ell}(a^{k}\om(d_{j+k}(v_{j+k},\tilde\alpha_{j+k}))) \leq \te/2\right\}.
$$
Then if $\te$ is small enough, since $v_{j,\omega}(\phi)\leq C$ for some constant $C$, for every $(u,v)\in \Gamma_{j,\ell,\te,x}$ we have 
$$
\sin^2\left(\frac12\left(t(S_{j,\ell}f(u,x)-S_{j,\ell}f(v,x))+\phi(u,x)-\phi(v,x)\right)\right)\geq \sin^2\left(\frac12\left(t(S_{j,\ell}f(\alpha,x)-S_{j,\ell}(\tilde \alpha,x))\right)\right)-\delta(\te)
$$
where $\delta(\te)\to 0$ as $\te\to 0$. Now, notice that by Proposition \ref{Ball ass2},
$$
(P_{j+\ell}(\cdot|x)\times P_{j+\ell}(\cdot|x))(\Gamma_{j,\ell,\te,x})\geq c(\ell,\te)>0
$$
for some $c(\ell,\te)$ which does not depend on $j$ or $x$.  
We thus conclude from \eqref{upp} that for every $x$ we have
$$
\sin^2\left(\frac12\left(t(S_{j,\ell}f(\al,x)-S_{j,\ell}(\tilde \al,x))\right)\right)\leq 2\varepsilon (c(\ell,\te))^{-1}+\delta(\te).
$$
Taking $\te=\te_\ell\to 0$ as $\ell\to\infty$ and then 
$\varepsilon=\varepsilon(\ell)$ small enough we get that there exists an integer valued function $Z_{j,\ell}$ such that 
$$
\sup_{x\in\cX_{j+\ell-1,\infty}}|S_{j,\ell}f(\al_j,...,\al_{j+\ell-1},T_{j+\ell-1}x)-S_{j,\ell}f(\al_j,...,\al_{j+\ell-2},x)-(2\pi/t)Z_{j,\ell}(x)|\leq c(\ell)
$$
where $c(\ell)\to0$ as $\ell\to\infty$. Namely,
$$
\sup|H_{j,\ell}\circ T_{j+\ell-1}-H_{j,\ell-1}-(2\pi/t)Z_{j,\ell}|\leq c(\ell).
$$
Next, let 
$$
\cD_{j,\ell}=H_{j,\ell}\circ T_{j+\ell-1}-H_{j,\ell-1}-(2\pi/t)Z_{j,\ell}.
$$
Then 
$$
\|\cD_{j,\ell}\|_\infty\leq c(\ell).
$$
Thus, since $\sup_{j,\ell}\|H_{j,\ell}\circ T_{j+\ell-1}-H_{j,\ell-1}\|_{j+\ell-1,\om}<\infty$ we must have that 
$$
\sup_{j,\ell}\|Z_{j,\ell}\|_{j+\ell-1,\om}<\infty
$$
since it is integer valued and can be approximated by $H_{j,\ell}\circ T_{j+\ell-1}-H_{j,\ell-1}$. Indeed, we have 
$$
|Z_{j,\ell}(x)-Z_{j,\ell}(y)|\leq c(\ell)+C\om_{j+\ell-1}(x,y), C>0
$$
and so if $\om_{j+\ell-1}(x,y)$ is smaller than $1-c(\ell)$ then $Z_{j,\ell}(x)=Z_{j,\ell}(y)$. Therefore, $|Z_{j,\ell}(x)-Z_{j,\ell}(y)|\leq C'\om_{j+\ell-1}(x,y)$ for some $C'>0$ and all $x,y$.
We thus conclude that 
$$
\sup_{j,\ell}\|\cD_{j,\ell}\|_{j+\ell-1,\om}<\infty.
$$
Using Lemma \ref{al be} it  follows that the decomposition  \eqref{Deccomp} holds with $H_{j,\ell-1}$ and $H_{j,\ell}$ instead of $H_{j+\ell-1}$ and $H_{j+\ell}$, respectively. This together with Lemma \ref{Aux2} completes the proof of the lemma.
\end{proof}

\section{Local limit theorems in the irreducible case}\label{SSKeyPropZN}
Building on the results in the previous sections the proof of Theorems \ref{LLT1} and \ref{LLT Latt} proceeds exactly like in \cite{DolgHaf LLT}, but we will still provide most of the details in order to make this paper self contained. In what follows we will write $\|g\|_{*}$ instead of $\|g\|_{j,*}$ when $g:\cX_{j,\infty}\to\bbC$.
 To simplify the proofs we will assume that $\ka_j(f_j)=0$ for all $j$, that is, $\bbE[S_nf]=0$ for all $n$. This could always be achieved by subtracting a 
constant from $f_j.$


\subsection{Contracting blocks}
\label{SSCBDef}
Fix some $T>\del$ and partition $\{|t|: \del\leq |t|\leq T\}$ into intervals  of small length $\del_1$ (yet to be determined).
 In order to prove \eqref{Suff} it is sufficient to show that if $\del_1$ is small enough then 
for every interval $J$ whose length is smaller than $\delta_1$ we have 

\begin{equation}\label{Suff J}
 \int_{J}\|\cL_{0, t}^n \|_{*}dt
=o(\sig_n^{-1}).
\end{equation}

Let us introduce  a simplifying notation. Given an interval 
in the positive integers $I=\{a,a+1,...,a+d-1\}$  we write
$
\cL_{t}^{I}=\cL_{a,t}^{d}=\cL_{a+d-1,t}\circ\cdots\circ\cL_{a+1,t}\circ \cL_{a,t}
$
and
$
 S_I=S_If=\sum_{j\in I}f_j\circ T_0^j.
$
Henceforth we will refer to a finite interval in the integers as a ``block". The length of a block is the number of integers in the block.

 Fix a small $\ve\in(0,1)$ (that will be determined latter).
We say that a block $I$ is {\em contracting} if 
\begin{equation}
\label{DefContrBl}
\sup_{t\in J}\|\cL_{t}^{I}\|_{*}\leq 1-\te(\ve)
\end{equation}
where $\te(\ve)>0$ comes from Corollary \ref{Corollary 4.9}. The following result is a version of \cite[Lemma 6.1]{DolgHaf LLT}.
\begin{lemma}\label{SubLemm}
If $I=\{a+1,a+2,...,a+d\}$ is  a non contracting block of size larger than $2k_0$ 
 (where $k_0$ comes from \eqref{Norm bound}) 
and $I''\subset I$ is a sub-block such that $I\setminus I''$ is composed of a union of two disjoint blocks whose lengths not less than $k_0$ 
then $I''$ is a non-contracting block.
\end{lemma}

Combining Lemmata  \ref{LmTDSmall} and \ref{SubLemm} together with Corollary \ref{Corollary 4.9} we
obtain the following result (see the proof of \cite[Corollary 6.2]{DolgHaf LLT}).

\begin{corollary}\label{CutLemm}
 Let $H_{k}$ be the functions   from Lemma \ref{LmTDSmall}. Let $k_0$ be 
 from
 \eqref{Norm bound} and $k_2(\cdot)$ be  from Corollary \ref{Corollary 4.9}.
If $\ve$ is small enough then 
there exists $L=L(\ve)\geq \max(k_0,2n_0+1)$ 
such that $L(\ve)\to \infty$ as $\ve\to 0$ with the following properties. 
If $I=\{a,a+1,...,a+d-1\}$ 
is  a non contracting block such that $d>2k_0+k_2(\ve)+L(\ve)$ 
then for every $s\in I$ with 
$$
a+k_0+L(\ve)+k_2(\ve)\leq s\leq a+d-k_0-1
$$
and all $t\in J$ we can write
\begin{equation}\label{Reduc}
tf_{s}=tg_s+tH_{s}-tH_{s+1}\circ T_s+2\pi Z_s
\end{equation}
where $ \|g_{s}\|_{s,\om^r}=O(\del_1)+ c_r(L)$ where $c_r(L)\to 0$  as $L\to \infty$ (here $\del_1$ is the length of the interval $J$  in \eqref{Suff J} and $r\in(0,1)$ is a fixed number). 
\end{corollary}
\begin{remark}
The functions $g_s$ and $Z_s$ can also depend on $t,\ve$ and $I$, but it is really important  for the next steps that the functions $H_s$ do not depend on $I$. 
\end{remark}

The last  key tool 
needed for the proof of \eqref{Suff J}  is the following  simple fact (see \cite[Lemma 6.4]{DolgHaf LLT}). 
\begin{lemma}\label{Qqad L}
 Let $Q(h)=ah^2+bh+c$ be a quadratic function with $a>0$ 
  and $\cJ$ be an interval.
  Then there is an absolute constant $C>0$ such that
 $$
 \int_{\cJ} e^{-Q(h)}dh\leq 
 \frac{C}{\sqrt{a}} \exp\left[-{  \min_{\cJ}} Q(h)\right].
 $$
\end{lemma}

Next, let 
$$
D(\ve)=4\left(L(\ve)+k_0+k_2(\ve)\right)
$$
where $L(\ve)$ comes from Corollary \ref{CutLemm}, $k_2(\ve)$ comes from Corollary \ref{Corollary 4.9} and $k_0$ comes from \eqref{Norm bound}. 
Let $L_n$ be the maximal number of contracting blocks contained in 
$I_n=\{0,1,...,n-1\}$, 
 such that 
the distance between consecutive blocks is at least $k_0$ and the 
length of each block is between $D(\ve)$ and $2D(\ve)$.

 Let $\mathfrak B\!\!=\!\!\{B_1, B_2, \dots, B_{L_n}\}$ be a corresponding set of contracting blocks
separated by at least $k_0$, 
and $\mathfrak A\!\!=\!\!\{A_1, \!\dots\!, A_{\tilde L_n}\}$ be a partition into intervals of the compliment of the union of the blocks $B_j$ in $I_n$, ordered so that $A_j$ is to the left of $A_{j+1}$. 
Thus
$\tilde L_n\!\!\in\!\! \{L_n\!-\!1, L_n, L_n\!+\!1\}.$

Like in \cite{DolgHaf LLT}
we will  prove of \eqref{Suff J} by considering three cases  depending on the size of $L_n.$

\subsection{Large number of contracting blocks}
\label{SSLrgeCB}

The first case is when $L_n$ is at least of logarithmic order in $\sig_n$. More precisely, we have the following result.

\begin{proposition}\label{case 1}
    There is a constant $c=c(\ve)>0$ such that \eqref{Suff J} holds if $L_n\geq c\ln\sig_n$. 
\end{proposition}

\begin{proof}
By the submultiplicativity of operator norms, Corollary \ref{Corollary 4.9} and 
\eqref{DefContrBl},
 $\forall t\in J$ we have
$$
\|\cL_{0,t}^n\|_*\leq \left(\prod_{k}\|\cL_{t}^{A_{k}}\|_*\right)\cdot \left(\prod_{j}\|\cL_{t}^{B_{j}}\|_*\right)\leq 
\left(1-\te(\ve)\right)^{L_n}.
$$
Here $\cL_t^{B}=\prod_{b\in B}\cL_{t,b}$.
Hence \eqref{Suff J} holds when $L_n\geq c\ln\sig_n$  for $c$ large enough. 
\end{proof}

\subsection{Moderate number of contracting blocks}

It remains to consider the case where $L_n\leq c\ln \sig_n$ for $c=c(\ve)$ from Proposition \ref{case 1}.

\begin{proposition}\label{case 2}
There is $\ve_0>0$ such that   \eqref{Suff J} holds if $\ve<\ve_0$ and $L_n\to\infty$ but
$L_n\leq c\ln \sig_n$,
   where $c=c(\ve)$ comes from Proposition \ref{case 1}.
\end{proposition}

We first need the following result, which is a version of \cite[Lemma 6.7]{DolgHaf LLT} 

\begin{lemma}\label{A lemma}
Let
$A=\{a,a+1,...,b\} \subset\{0,1,...,n-1\}$ be a block of length greater or equal to $4D(\ve)+1$, which 
does not intersect contracting blocks from $\mathfrak{B}$. Define $a'=a+2 k_0+L(\ve)+k_2(\ve)$ and 
$b'=b-2k_0-1$, and set $A'=\{a',a'+1,...,b'\}$. Then for every 
$s\in A'$ and all $t\in J$ we can write 
\begin{equation}\label{deccomp}
    tf_s=g_{t,s}+H_s-H_{s+1}\circ T_s+2\pi Z_s
\end{equation}
where $\|g_{t,s}\|_{j,\om^r} \leq C_r(\ve)+c_0\del_1$ for some $C_r(\ve)$ such that $C_r(\ve)\to 0$ as $\ve\to 0$ (and $r\in(0,1)$ is some fixed number),  $c_0$ is a constant (here $\del_1$ is the length of the underlying interval $J$) and $Z_s$ are integer valued.
\end{lemma}

\begin{proof}
 Let $s\in A'.$ Then, since $\mathfrak{B}$ is maximal, the block of length $D(\ve)$ ending at $s$
is non contracting. Therefore the result follows from Corollary \ref{CutLemm}.
\end{proof}

We also need the following result, whose proof is proceeds exactly like in \cite[Lemma 6.8]{DolgHaf LLT}
\begin{lemma}\label{CombLemma}
Suppose that $L_n=o\left(\sig_n^2\right)$. Then 
 there exists $ m=m_n\in \{1, \dots, \tilde L_n\}$ such that
\vskip0.2cm

(i) For all $n$ large enough we have 
$
\|S_{A_{m}}\|_{L^2}\geq \frac{\sig_n}{4\sqrt{L_n}}$ $ \left(\text{and so } |A_{m}|\geq \frac{\sig_n}{4\sqrt{L_n}\|f\|_\infty} \right)
$
where $ \|f\|_\infty=\sup_j\|f_j\|_\infty$ and $|A_m|$ is the size of $A_m$.
\vskip0.2cm

(ii)
Write $A_{m_n}=\{a_n,a_{n}+1,..., b_n\}$ and set 
$A_{m_n}'=\{a_n',a_n'+1,...,b_n'\}$, where 
$$a_n'=a_n+ 2 k_0+L(\ve)+k_2(\ve)\,\, \text{ and }\,\,
b_n'=b_n- 2k_0-1.
$$
Then, if $n$ is large enough, \eqref{deccomp} holds for every $s\in A_{m_n}'$ and all $t\in J$.

\end{lemma}

\begin{proof}
Let $\mathfrak C=\mathfrak A\cup \mathfrak B$. Then 
$
S_n=\sum_{ C\in \mathfrak C}S_{C}=\sum_{k=1}^{L_n}S_{B_k}+\sum_{l=1}^{\tilde L_n}S_{A_l}.
$
Thus
$$ 
\sig_n^2=\|S_n\|_{L^2}^2=\sum_k \|S_{B_k}\|_{L^2}^2+\sum_l \|S_{A_l}\|_{L^2}^2+2\sum_{l_1<l_2} \mathrm{Cov}(S_{C_{l_1}}, S_{C_{l_2}}).
$$ 
Now, since the size of each block $B_j$ is at most $2D(\ve)$ and  $ \|f\|_\infty=\sup_k\|f_k\|_\infty<\infty$,
$$
\sum_{1\leq k\leq L_n}\|S_{B_k}\|_{L^2}^2\leq \left(2D(\ve)\|f\|_\infty\right)^2 L_n
=o(\sig_n^2).
$$
Next,  for every sequence of random variables $(\xi_j)$ such that $|\text{Cov}(\xi_n,\xi_{n+k})|\leq C\del^k$ for $C>0$ and $\del\in(0,1)$ we have the following.  For all $a<b$ and $k,m>0$,
$$
|\text{Cov}(\xi_{a}+...+\xi_{b}, \xi_{b+k}+...+\xi_{b+k+m})|\leq \sum_{j=a}^b\sum_{s=k}^{\infty}|\text{Cov}(\xi_j,\xi_{b+s})|\leq C\sum_{j=a}^b\sum_{s\geq k}\del^{b+s-j}
$$
$$
\leq C_\del\sum_{j=a}^b\del^k\del^{b-j}=C_{\del}\del^k\sum_{j=0}^{b-a}\del^j\leq C_\del(1-\del)^{-1}
$$
where $C_\del=C/(1-\del)$.
Applying this with $\xi_j=f_j(X_j,X_{j+1},...)$ (and using the exponential decay of correlations that follows from Theorem \ref{RPF}), we see that there is a constant $R>0$ such that
for each $l_1$, 
$$\left| \sum_{C_{l_2}: l_2>l_1} \mathrm{Cov}\left(S_{C_{l_1}}, S_{C_{l_2}}\right)\right|\leq R.$$
Therefore
$$
\sig_n^2=\sum_{1\leq k\leq  \tilde L_n}\|S_{A_k}\|_{L^2}^2+o(\sig_n^2)+O(L_n)=
\sum_{1\leq k\leq  \tilde L_n}\|S_{A_k}\|_{L^2}^2+o(\sig_n^2).
$$
Thus, if $n$ is large enough then 
there is at least one index $m$ such that 
$$
\|S_{A_m}\|_{L^2}\geq \frac{\sig_n}{4\sqrt{L_n}}.
$$
Next,   by the triangle inequality
$
 \|S_{A_{m}}\|_{L^2}\leq \
\sup_{j}\|f_j\|_{L^2(\mu_j)}|A_m|\leq \|f\|_\infty |A_m|
$
and so
$
|A_m|\!\!\geq\!\! \frac{\sig_n}{4{\sqrt{L_n}}\|f\|_\infty}.
$
Thus property (i) holds.
Property (ii) follows from  Lemma~\ref{A lemma}.
\end{proof}

   To complete the proof of Proposition \ref{case 2}, we will prove the following result.

  \begin{lemma}\label{Norm Lem}
  There is $\ve_0>0$ such that if the length $\del_1$ of $J$ satisfies $\del_1<\ve_0$ and if $\ve<\ve_0$ then
      $
\int_{J}\|\cL_{t}^{A_{m_n}}\|_*dt=O\left(\sqrt{L_n}\sig_n^{-1}\right).
      $
  \end{lemma}

\begin{proof}[Proof of Lemma \ref{Norm Lem}]
Let $A_{m_n}'$ be defined in Lemma \ref{CombLemma}. 
Then we can write 
$$
A_{m_n}=U_n\cup A_{m_n}'\cup V_n
$$
for blocks $U_n$ and $V_n$ such that  $U_n, A_{m_n}', V_n$ are disjoint, $U_n$ is to the left of $A_{m_n}'$ and $V_n$ is to its right. Moreover, $V_n$ is of size $2k_0+1$ and $U_n$ is of size $2k_0+L(\ve)+k_2(\ve)$. Thus by \eqref{Norm bound},
$$
\sup_{t\in J}
\max\left(\|\cL_t^{U_n}\|_*,\|\cL_t^{V_n}\|_* \right)\leq 1.
$$
Since 
$
\cL_{t}^{A_{m_n}}=\cL_{t}^{V_n}\circ\cL_{t}^{A_{m_n}'}\circ\cL_t^{U_n}
$
we conclude that 
$
\|\cL_{t}^{A_{m_n}}\|_*\leq \|\cL_{t}^{A_{m_n}'}\|_*.
$
Thus, its enough to show that 
\begin{equation}\label{suff A prime}
\int_{J}\|\cL_{t}^{A_{m_n}'}\|_*dt=O\left(\sqrt{L_n}\sig_n^{-1}\right).
\end{equation}

Next, let us write $A_{m_n}'=\{a'_n,a'_{n}+1,...,b'_n\}$. 
Then,  by Lemma \ref{CombLemma}(ii), for all $s\in A'_{m_n}$ and all $t\in J$ we can write 
$
tf_s=g_{t,s}+H_s-H_{s+1}\circ T_s+2\pi Z_{t,s}
$
where $\|g_{t,s}\|_{\al} \leq C(\ve)+c_0\del_1$, for some $C(\ve)$ such that $C(\ve)\to 0$ as $\ve\to 0$, and $c_0$ is a constant. 
In particular, if $t_0$ is the center of $J$ and $t=t_0+h\in J$, then for all $s\in A_{m_n}'$ we have 
$$
tf_s=t_0f_s+hf_s=(g_s+hf_s)+t_0H_s-t_0H_{s+1}+2\pi Z_s
$$
where $g_s=g_{t_0,s}$ and $Z_s=Z_{t_0,s}$.
Therefore, for any function $u$ we have
$$
\cL_{t}^{A_{m_n}'}u=e^{-it_0H_{b'_{n}}}\cL_{a'_{n}}^{b'_{n}-a'_{n}}(e^{iS_{a'_{n},b'_{n}}g+it_0H_{a'_{n}}+
ihS_{a'_{n},b'_{n}}f}u).
$$
Let 
$ 
\cA_{t}(u):=\cL_{a'_{n}}^{b'_{n}-a'_{n}}(e^{iS_{a'_{n},b'_{n}}g+ihS_{a'_{n},b'_{n}}f}u).
$
Then since $ \sup_j\|H_j\|_{\al}<\infty$ there is a constant $C>0$ such that 
$$
\|\cL_{t}^{A_{m_n}'}\|_{*}\leq C\|\cA_t\|_*.
$$
 Now,  by Proposition \ref{PrSmVarBlock} 
there exist
  constants $C_0>0$ and $c>0$ such that  if  $\del_1$ (and hence $|h|$) and $\ve$ are small enough
then 
$$\|\cA_t\|_*\leq C_0 e^{-cV_{n}(h)}$$
where $V_n(h)=\|S_{a_n',b_n'}(g+hf)\|_{L^2}^2$.
Thus
$$
\int_{J}\|\cL_{t}^{A_{m_n}'}\|_{*}dt=\int_{-\del_1/2}^{ \del_1/2}
\|\cL_{t_0+h}^{A_{m_n}'}\|_{*}dh\leq CC_0''\int_{-\del_1/2}^{ \del_1/2}
e^{-cV_{n}(h)}dh.
$$
Applying Lemma  \ref{Qqad L} with $Q(h)\!=\!V_n(h)\!=\!\|S_{a_{n},b_{n}}(g+hf)
\|_{L^2}^2$ and 
using Lemma~\ref{CombLemma}(i) we conclude that there  are constant $C',C''>0$ such that
$$
\int_{-\del_1/2}^{\del_1/2}\|\cL_{0,t_0+h}^{A_{m_n}'}\|_{*}dh\leq C'\left(\text{Var}(S_{a'_n,b'_n})\right)^{-1/2}\leq C''L_n^{{1/2}}\sig_n^{-1}
$$
where we have used that the quadratic form $ Q(h)$ is nonnegative, and \eqref{suff A prime} follows.
\end{proof}

\begin{proof}[Proof of Proposition \ref{case 2}]
By the submultiplicativity of the norm and Corollary \ref{Corollary 4.8}, 
$$ \|\cL_{0,t}^{n}\|_*\leq \left(\prod_{k=1}^{\tilde L_n} \|\cL_t^{A_k}\|_*\right)\left(\prod_{j=1}^{L_n} \|\cL_t^{B_j}\|_*\right)
\leq (1-\te(\ve))^{L_n} \|\cL_t^{A_{m_n}}\|_*.
$$
  Thus Lemma \ref{Norm Lem} gives
  $
\int_{J}\|\cL_{0,t}^{n}\|dt\leq (1-\eta(\ve))^{L_n} {\sqrt{L_n}}\sig_n^{-1}
  $
  which is indeed $o(\sig_n^{-1})$ if $L_n$  diverges to infinity. 
\end{proof}

\subsection{Small number of contracting blocks}
\label{SSFewBlocks}
The third and last case we need to cover to complete the proof of \eqref{Suff J} is when $L_n$ is bounded.
\begin{proposition}\label{case 3}
 If $L_n$ is bounded then either 
\vskip0.1cm  
(i) $(f_j)$ is reducible  and, moreover, one can decompose
$$ f_{j}=g_{j}+H_{j-1}-H_{j}\circ T_{j-1}+(2\pi/t) Z_{t,j} $$
 with $t\in J$,\;\; $Z_{t,j}$ integer valued and  $g_j\circ T_0^j$ is a reverse martingale difference satisfying
 $ \sum_j \mathrm{Var}(g_j)<\infty$;
 or
 \vskip0.1cm  
(ii) 
$
\int_{J}\|\cL_{0,t}^n\|_*dt=o(\sig_n^{-1}).
$
\end{proposition}

\begin{proof}

Suppose that $L_n$ is bounded, and let $N_0$ be the right end point of the last contracting block $B_{L_n}$. If $L_n=0$ we set $N_0=-1$. Set $N(\ve)=N_0+k_0+L(\ve)+k_2(\ve)+1$.
Then, for every $t\in J$ and $n\geq N(\ve)$ we have
$$
\|\cL_{0,t}^n\|_*=\|\cL_{N(\ve),t}^{n-N(\ve)}\circ\cL_{0,t}^{N(\ve)}\|_*\leq \|\cL_{N(\ve),t}^{n-N(\ve)}\|_*\|\cL_{0,t}^{N(\ve)}\|_{*}\leq \|\cL_{N(\ve),t}^{n-N(\ve)}\|_*
$$
where in the second inequality we have used \eqref{Norm bound}.

Now, by Lemma \ref{A lemma} applied with $A=\{N_0,...,n-1\}$ we see that  
there  are functions  $g_j,\tilde H_j, \tilde Z_j$ such that for all $j\geq N(\ve)$ we have
\begin{equation}\label{Redd}
t_0f_j=g_j+\tilde H_j-\tilde H_{j+1}\circ T_j+2\pi \tilde Z_j
\end{equation}
 where $t_0$ is the center of $J$ and $ \sup_j\|g_j\|_{\alpha}\leq C(\ve)$, with $C(\ve)\to 0$ as $\ve\to 0$. Hence, like in the previous cases, if we take $\ve$ and $\del_1$ small enough, then 
 it is enough  to bound the norm of the  operator $\cA_{n,t}$ given by
$$ 
\cA_{n,t}u=\cA_{n,t_0+h}u=\cL_{N(\ve)}^{n-N(\ve)}(e^{iS_{N(\ve),n-N(\ve)}(g+hf)}),
\text{ where } h=t-t_0. 
$$
Let $\del_0$  be the constant from Proposition \ref{PrSmVarBlock}.
Take $\ve$ and $\del_1$ small enough so that $ \sup_{|h|\leq\del_1}\sup_j\|g_j+hf_j\|_{\alpha}<\del_0$. 
 Applying Proposition \ref{PrSmVarBlock}  we have 
$$
\|\cA_{n,t_0+h}\|_*\leq C e^{-cQ_n(h)}$$
for some constant $c>0$
where, as before 
$$
 Q_n(h)=\text{Var}(\tilde S_n f)h^2+2h\text{Cov}(\tilde S_nf, \tilde S_n g)+\text{Var}(\tilde S_ng)
$$
where $\tilde S_nf=S_{N(\ve),n-N(\ve)}f$ and $\tilde S_n g=S_{N(\ve),n-N(\ve)}g$. Then $\tilde\sig_n:=\|\tilde S_nf\|_{L^2}\geq \sig_n-CN(\ve)$ for some constant $C$.
Thus, by Lemma \ref{Qqad L} there is a constant $A>0$ such that
\begin{equation}\label{l64}
    \int_{-\del_1/2}^{\del_1/2}\|\cA_{n,t_0+h}\|_*\leq 
A\sig_n^{-1} \exp\left(-c\min_{[-\delta_1/2, \delta_1/2]} Q_n(h)\right).
\end{equation}
Note that
$$
m_n=\min Q_n=\text{Var}(\tilde S_n(g+h_nf)) \geq 0
$$
where $h_n:=\mathrm{argmin 
 }Q_n=-\frac{\mathfrak b_n}{2\mathfrak a_n}$.
Thus, if $m_n\!\!\to\!\!\infty$ then \eqref{Suff J} holds. Let us suppose that 
$ \liminf_{n\to\infty} m_n<\infty.$ 
We claim that in this case either \eqref{Suff J} holds or $(f_j)$ is reducible to a lattice valued sequence of functions.
Before proving the claim, let us simplify the notation and write
$$
Q_n(h)=\sig_n^2(h-h_n)^2+m_n=\mathfrak a_n h^2+\mathfrak b_n h+\mathfrak c_n
$$
where $\mathfrak a_n=\tilde\sig_n^2.$ Thus $h_n=\mathrm{argmin 
 }Q_n=-\frac{\mathfrak b_n}{2\mathfrak a_n}$. 

We now consider two cases.
\vskip0.1cm
(1) For any subsequence  with $ \lim_{j\to\infty} m_{n_j}<\infty$ we have
 $|h_{n_j}|\geq \del_1$. Then 
 $$ \min_{[-\delta_1/2, \delta_1/2]} Q_{n_j}(h)\geq \frac{\mathfrak a_{n_j}\delta_1^2}{4}
  =\frac{
  \tilde\sig_{n_j}^2\delta_1^2}{4}
 $$
and so \eqref{Suff J} holds  by \eqref{l64}.
\vskip0.1cm
(2) It remains to consider the case when there is a subsequence 
$n_j$ such that $|h_{n_j}|\leq \delta_1$ and $\bar Q:=\lim_{j\to\infty} m_{n_j}<\infty.$
By taking further subsequence if necessary we may assume that the limit
$\lim_{j\to\infty} h_{n_j}= h_0$ exists. Then for all $n$,
$$ Q_n(h_0)=\lim_{j\to\infty} Q_{n}(h_{n_j})=
\lim_{j\to\infty} \text{Var}(\tilde S_{n}(g+h_{n_j} f))
$$
$$
=\lim_{j\to\infty} \left(\text{Var}(\tilde S_{n_j}(g+h_{n_j} f))-\text{Var}(S_{n, n_j-n}f)
-2\text{Cov}(\tilde S_n (g+h_{n_j} f), S_{n, n_j-n}(g+h_{n_j} f))\right)
$$$$
\leq \lim_{j\to\infty} m_{n_j}-2\lim\inf_{j\to\infty} \text{Cov}(\tilde S_n (g+h_{n_j} f), S_{n, n_j-n}(g+h_{n_j} f))\leq \bar Q+2C
$$ 
for some constant $C>0$, where the last inequality uses that $\left|\text{Cov}(f_j, f_{j+k}\circ T_j^k)\right|\leq c_0\del_0^k$ for some constants $c_0>0$ and $\del_0\in(0,1)$.
Since $Q_n(h_0)\leq \bar Q+2C$ for all $n$ we obtain
$$ \limsup_{n\to\infty} \text{Var}(\tilde S_n(g+h_0 f))\leq \bar Q+2C.$$
Hence $ \sup_n\text{Var}(S_n(g+h_0 f))<\infty$.
Thus, by Theorem \ref{Var them}   we can write
$$
h_{0}f_j+g_j=\mu_j(h_{0}f_j+g_j)+M_j+u_{j+1}\circ T_j-u_j
$$
with functions $u_j$ and  $M_j$ such that $ \sup_j\|u_j\|_{j,\om^r}, 0<r<1$ and $ \sup_{j}\|M_j\|_{j,\om^r}$ are finite, $M_j(X_j,X_{j+1},...)$ is a reverse martingale difference with 
$ \sum_j\text{Var}(M_j)<\infty$. Combining this with \eqref{Redd} we conclude that 
$(t_0+h_0)(f_j)$ is reducible to a $2\pi \bbZ$ valued sequence of functions,
and the proof of Proposition \ref{case 3} is complete.
\end{proof}

\subsection{Proof of the main results in the irreducible case}
\label{Edge1}

 Combining the results of \S\S \ref{SSLrgeCB}--\ref{SSFewBlocks} we obtain \eqref{Suff} completing the proof of  Theorem \ref{LLT1}. \vskip1mm

To prove Theorem \ref{LLT Latt} we note that the analysis 
of \S\S \ref{SSLrgeCB}--\ref{SSFewBlocks}  (in particular the proof of 
Proposition \ref{case 3}) also shows that if $\cJ$ is an interval such that 
 $\int_{\cJ} \|\cL_{0, t}^n \|_{*}dt \neq 0$ then $(f_j)$ is reducible to $h\bbZ$
valued sequence for some $h$ with $\frac{2\pi}{h} \in \cJ.$
By the assumption of Theorem \ref{LLT Latt} such a reduction is impossible for $|h|>1$
(see Theorem \ref{Reduce thm}) and therefore \eqref{SuffLat} holds, implying Theorem 
\ref{LLT Latt}.
\vskip1mm

 Theorem \ref{ThEdge1} follows by the estimates at the end of \cite[Section 6.6.1]{MarShif1}, the estimate \eqref{Suff} and \cite[Proposition 25]{DolgHaf PTRF 1} with $r=1$.

\section{Local limit theorem in the reducible case}\label{Red sec}
Like in \cite[Section 8]{DolgHaf LLT} we can prove the following results (the proof is almost identical to \cite[Section 8]{DolgHaf LLT} so it is omitted).

Suppose that $(f_j)$ is a reducible sequence such that $\sig_n\to\infty$. 
Let
$a=a(f)$ be the largest positive number such that $f$ is reducible to an $a\bbZ$ valued sequence
 (such $a$ exists by Theorem \ref{Reduce thm}).  Let $\del=2\pi/a$ and  write 
\begin{equation}
\label{RedF}
\del f_j=\del\mu_j(f_j)+M_j+g_j-g_{j+1}\circ T_j+2\pi Z_j
\end{equation}
with $(Z_j)$ being an integer valued irreducible sequence, $g_j,M_j$ are functions 
such that
$ \sup_{j}\|g_j\|_\al<\infty$, $ \sup_j\|M_j\|_\al<\infty$, and $(M_j\circ T_0^j)$ is a reverse martingale difference with respect to the reverse filtration $(T_0^j)^{-1}\cB_j$ on the probability space $(X_0,\cB_0, \ka_0)$. 
Moreover, we have 
$
\sum_{j}\text{Var}(M_j)<\infty.
$
Then by the martingale convergence theorem the limit $ \textbf{M}=\lim_{n\to\infty}S_n M$ exists. Set $\textbf{A}=\textbf{M}+g_0$. 
\begin{theorem}\label{LLT RED}
If \eqref{RedF} holds then
for every continuous function $\phi:\bbR\to\bbR$ with compact support, 
$$
\sup_{u\in a\bbZ}\left|\sqrt{2\pi}{ \sig_n}\bbE_{\ka_0}[\phi(S_n\!-\!u)]
\!\!-\!\!
\left(\!\! a\sum_{k}\int_{X_n}\!\!\bbE_{\ka_0}[\phi(k a+\textbf{A}\!-\!g_n(x))]
d\mu_n(x)\!\!\right)\!e^{-\frac{(u-\bbE[S_n])^2}{2\sig_n^2}}\right|
\!\!=\!\! o(1).
$$
\end{theorem}
We refer to \cite[Remark 7.2]{DolgHaf LLT} for several comments and motivation concerning Theorem \ref{LLT RED}.

\section{Irreducibly for connected spaces}\label{Irr Sec}

Here we prove Theorem \ref{LLT2}. As we have explained before, it suffices to prove \eqref{Suff}. Hence Theorem \ref{LLT2} follows from the
estimate below, whose proof proceeds exactly like the proof of \cite[Proposition 9.7]{DolgHaf LLT}.

\begin{proposition}\label{ExpDecP stretched}
 If all the spaces $\cX_j$ are  connected 
then for every $0<\del<T$ there are constants $c,C>0$ such that for all $n$,
\begin{equation}\label{expp}
\sup_{\del\leq |t|\leq T}\|\cL_{0,t}^n\|_{*}\leq Ce^{-c\sig_n}.
\end{equation}
Moreover, if $\sig_n\to \infty$ then $(f_j)$ is irreducible.
\end{proposition}

\section{Two sided sequences and functions via Sinai's lemma}
\subsection{Sinai's lemma}
Let $(X_j)_{j\in\bbZ}$ be a two sided Markov chain satisfying Assumptions \ref{Assss} and \ref{Ball Ass}.
Denote $\cY_j=\prod_{k\in\bbZ}\cX_{j+k}$ and for $x,y\in\cY_j$
$$
\om_{j,a}(x,y)=2A_1\sup_{n\in\bbZ}a^{|n|}\omega(d_{j+n}(x_{j+n},y_{j+n}))
$$
for $a\in(0,1)$.
In this section we explain how to derive all the results from the previous section for partial sums of the form 
$$
S_nf=\sum_{j=0}^{n-1}f_j(...,X_{j-1},X_j,X_{j+1},...)
$$
from the corresponding results in the case when $f$ depends only on $X_k, k\geq j$. This is based on the following version of Sinai's lemma.

\begin{lemma}\label{Sinai}
Let $f_j:\cX_{j,\infty}\to\bbR$ be such that $\sup_j\|f_j\|_{j,a,\om}<\infty$. Then there exist functions $u_j:\cY_j\to\bbR$ and $g_j:\cX_{j,\infty}\to\bbR$ such that $\sup_j\|u_j\|_{j,a^{1/3},\om^{1/3}}<\infty$ and $\sup_j\|g_j\|_{j,a^{1/3},\om^{1/3}}<\infty$
and
$$
f_j(x)=u_{j+1}(...,x_{j},x_{j+1},x_{j+2},...)-u_j(...,x_{j-1},x_j,x_{j+1},...)+g_j(x_j,x_{j+1},...).
$$
\end{lemma}
\begin{remark}
It will be clear that from the proof that $g_j=f_j$ when $f_j$ depends only on the coordinates $x_{j+k},k\geq 0$ and then $u_j=0$.
\end{remark}

\begin{proof}[Proof of Lemma \ref{Sinai}]
Fix some point $\alpha\in\prod_{k<j}\cX_k$ and 
  $u_j$ by
 $$
u_j(x)=\sum_{k=0}^\infty \left(f_{j+k}(x)-f_{j+k}(\alpha,x_{j},x_{j+1}....)\right).
$$
 Then for some constant $A>0$ we have
 \begin{equation}\label{up}
  |f_{j+k}(x)-f_{j+k}(\alpha,x_{j},x_{j+1}....)|\leq A a^{k}   
 \end{equation}
 and so 
 $$
\|u_j\|_{L^\infty}\leq \sum_{k=0}^{\infty}a^{-k}v_{\om,j+k}(f_{j+k})\leq C<\infty
$$
for some constant $C>0$.
 Notice that $u_j-u_{j+1}\circ T_j-f_j$ depends only on the coordinates with indexes $j+k,k\geq 0$, where $T_j:\cX_{j,\infty}\to\cX_{j+1,\infty}$ is the left shift. Set $g_j=f_j+u_{j+1}\circ T_j-u_j$.

 In order to complete the proof of the lemma it is enough to show that $\sup_j v_{j,a^{1/3},\om^{1/3}}(u_j)<\infty$. Let us take some $x,y\in\cY_j$ and suppose that
 $$
\om_j(x,y)=2A_1a^{n_0}\omega_{j+n_0}(d_{j+n_0}(x_{j+n_0},y_{j+n_0}))
 $$
 for some $n_0\geq 0$ (we assume here that $n_0\geq0$ and the argument for $n_0<0$ is similar). Let us take $m_0$ such that $(2A_1)a^{m_0}\leq \om_j(x,y)<(2A_1)a^{m_0+1}$. In particular, $m_0\leq n_0$. 
 Then using \eqref{up} $k\geq m_0/2$ we see that there is a constant $C_1>0$ such that
 $$
|u_j(x)-u_j(y)|\leq \sum_{k<m_0/2}|f_{j+k}(x)-f_{j+k}(y)|+ \sum_{k<m_0/2}|f_{j+k}(\alpha,x_j,x_{j+1},...)-f_{j+k}(\alpha,y_{j},y_{j+1},...)|+C_1(\om_j(x,y))^{1/2}.
 $$
 Now, for $k<m_0/2$ we have 
 $$
|f_{j+k}(x)-f_{j+k}(y)|\leq C\sup_{m}a^{|m|}\om(d_{j+k+m}(x_{j+k+m},y_{j+k+m})).
 $$
Suppose that for some $m_k\geq 0$ we have  
 $$
\sup_{m}a^{|m|}\om(d_{j+k+m}(x_{j+k+m},y_{j+k+m}))=a^{m_k}\om(d_{j+k+m_k}(x_{j+k+m_k},y_{j+k+m_k})).
 $$
 Then since  $\om_j(x,y)=(2A_1)a^{n_0}\omega_{j+n_0}(d_{j+n_0}(x_{j+n_0},y_{j+n_0}))$ and $k<m_0/2\leq n_0/2$
 $$
m_0 a^{m_k}\om(d_{j+k+m_k}(x_{j+k+m_k},y_{j+k+m_k}))=m_0a^{-k}a^{m_k+k}\om(d_{j+k+m_k}(x_{j+k+m_k},y_{j+k+m_k}))
$$
$$
\leq m_0 a^{n_0-k}\omega_{j+n_0}(d_{j+n_0}(x_{j+n_0},y_{j+n_0}))
 $$
 $$
\leq Ca^{n_0/3}\omega_{j+n_0}(d_{j+n_0}(x_{j+n_0},y_{j+n_0}))\leq Ca^{n_0/3}\left(\omega_{j+n_0}(d_{j+n_0}(x_{j+n_0},y_{j+n_0}))\right)^{1/3}=C'(\om_j(x,y))^{1/3}. 
 $$
 If for some $m_k<0$ we have
 $$
\sup_{m}a^{|m|}\om(d_{j+k+m}(x_{j+k+m},y_{j+k+m})=a^{-m_k}\om(d_{j+k+m_k}(x_{j+k+m_k},y_{j+k+m_k}))
 $$
 then 
 $$
a^{-m_k}\om(d_{j+k+m_k}(x_{j+k+m_k},y_{j+k+m_k}))=a^ka^{-m_k-k}\om(d_{j+k+m_k}(x_{j+k+m_k},y_{j+k+m_k}))\leq a^k\om_j(x,y).
 $$
Arguing similarly with $|f_{j+k}(\alpha,x_j,x_{j+1},...)-f_{j+k}(\alpha,y_{j},y_{j+1},...)|$ instead of $|f_{j+k}(x)-f_{j+k}(y)|$ and using that $a\in(0,1)$ and that $\om$ takes values in $[0,1]$ we see that there exists a constant $C_2>0$ such that
 $$
|u_j(x)-u_j(y)|\leq C_2(\om_j(x,y))^{1/3}.
 $$
\end{proof}

\subsection{Reducibility in the  two sided case}

Let $\mu_j$ be the law of $(X_j)_{j\in\bbZ}$ induced on $\cY_j$. Of course, both $\mu_j$ and $\cY_j$ do not really depend on $j$, but we will still distinguish between them since we denote by $x_{j+k}$ th $k$-th coordinate  of $x=(x_{j+k})_{k\in\bbZ}\in\cY_j$.  

\begin{assumption}\label{CondAss}
There exists a constant $C>0$ such that for every $j$ and all measurable $\Gamma\subset\cX_{j}$ and $x\in\cX_{j-1}$ we have 
$$
\bbP(X_j\in\Gamma|X_{j-1}=x)\leq C\bbP(X_j\in\Gamma).
$$
Moreover, the density $p_j(x,y)$ of $X_{j+1}$ given $X_j=x$ satisfies $|p_j(x,y)-p_j(x',y)|\leq C\om(d_{j}(x,x'))$.
\end{assumption}
\begin{example}
The first part of assumption holds if the first upper $\psi$-mixing coefficient is finite, namely if
$$
\psi_U(1)=\sup_j\sup\left\{\frac{\bbP(A\cap B)}{\bbP(A)\bbP(B)}-1: A\in\cF_{-\infty,j}, B\in\cF_{j+1,\infty}. \bbP(A)\bbP(B)>0\right\}<\infty.
$$
\end{example}
\begin{remark}\label{Remmm}
Under Assumption \ref{CondAss} we have that for every measurable set $\Gamma\subset\cX_{j,\infty}$ we have 
$$
\bbP((X_{j+k})_{k\geq0}\in\Gamma|X_{j-1}=x)\leq C\bbP((X_{j+k})_{k\geq0}\in\Gamma).
$$
Note that this also follows from the assumption that $\Psi_U(1)<\infty$.
Moreover, the condition  $|p_j(x,y)-p_j(x',y)|\leq C\om(d_{j+1}(x,x'))$ implies that the density $p(x,\bar y)$ of $(X_{k})_{k\geq j+1}$ given $X_j=x$ satisfies 
$|p_j(x,\bar y)-p_j(x',\bar y)|\leq C'\om(d_{j}(x,x'))$ for constant $C'>0$.
\end{remark}

\begin{proposition}\label{Red2sided}
Let Assumptions \ref{MixAss}, \ref{Assss}, \ref{Ball Ass} and \ref{CondAss} be in force.
 Let $g_j: \cX_j\to\bbR$ be the functions like in Lemma \ref{Sinai}. Then for every $h>0$ we have that $(g_j)$ is reducible to an $h\bbZ $-valued sequence iff $(f_j)$ is reducible to an $h\bbZ $-valued sequence.
\end{proposition}
\begin{proof}
First, it is clear that $(f_j)$ is reducible if $(g_j)$ is. Conversely, suppose that
 $(f_j)$ is reducible. 
 Then there are $h\!\!\neq\!\!0$ and 
 functions $H_j\!\!:\!\! \tilde X_j\!\!\to\!\! \mathbb{R}$, 
 $Z_j\!\!:\!\! \tilde X_j\!\!\to\!\! \mathbb{Z}$  
such that $\sup_j
\|H_j\|_{\om},<\infty$, $(S_n H)_{n=1}^\infty $ is  tight
and $f_j=H_j+hZ_j$ for all $j$. 
Applying Theorem \ref{Var them} 
with the sequence $(H_j)$ on the two sided shift (which is possible in view of Lemma \ref{Sinai})
  we can decompose
 $
 f_j=u_j-u_{j+1}\circ T_j+M_j+h Z_j,
 $
 where $M_j$ is a reverse martingale difference and  $\sup_j\max(\|u_j\|_{j,\om},\|M_j\|_{j,\omega})<\infty$. Moreover $T_j:\cY_j\to\cY_{j+1}$ is the left shift and
 $\sum_{j}\text{Var}(M_j)<\infty$.
Now, since $M_j$ is a reverse martingale difference and 
$\sum_{j}\text{Var}(M_j)<\infty$ 
we have that  with probability 1,
$S_{j,n}M$ can be made  arbitrarily small for large $j$. 
 Thus, by the Dominated Convergence Theorem we can ensure that
 $\bbE[e^{it S_{j,n}M }]$ is arbitrarily close to $1$ as $j\to\infty$,
 where $t=2\pi/h$. Now,  assume for the sake of contraction that $(g_j)$ is irreducible.  Then,  like in the proof of Corollary \ref{H corr} we see that the operator
 norms of $\cL_{j,t}^n$ with respect to the norms induces by $a^{1/3}$ and $\om^{1/3}$ decay to $0$  as $n\to\infty$
  for every nonzero $t$, where $\cL_{j,t}(h)=\cL_j(he^{itg_j})$.
Next, we show that under this assumption for every $j$ we get that  $\bbE[e^{it S_{j,n}M }]\to 0$ as $n\to\infty$, which contradicts that   $\bbE[e^{it S_{j,n}M }]$ 
is close to $1$. This will complete the proof. 
In order to prove that $\bbE[e^{it S_{j,n}M }]\to 0$ as $n\to\infty$, let us first note that $M_j=g_j+v_{j+1}-v_{j}-hZ_j$ for some sequence of functions $v_j$ with 
$\sup_j\|v_j\|_{a^{1/3},\om^{1/3}}<\infty$.

Now, by conditioning on the past $X_k, k<j$ we have
$$
\bbE[e^{itS_{j,n} M}]=
\bbE[e^{it S_{j,n}g+it v_{j+n}\circ T_j^n-it v_j}]
=\int\left(\int e^{itS_{j,n}g(x)+itv_{j+n}(T_j^n(y^{-},x))+itv_j(y^{-},x)}p_j(x|y_{j-1})d\kappa_j(x)\right)d\nu_j(y^{-})
$$
where $y^{-}=(...,y_{j-2},y_{j-1})$ and $p_j(x|y_{j-1})$ is the density of $(X_{j+k})_{k\geq 0}$ given $X_{j-1}=y_{j-1}$
and $\nu_j$ is the law of $(X_{k})_{k<j}$. 
 Next, since $\ka_j=(\cL_j^n)^*\ka_{j+n}$
for every realization $y^{-}$ we have 
$$
\int_{\cX_{j,\infty}} e^{itS_{j,n}g(x)+itv_{j+n}(\tilde T_j^n(y^{-1},x))+itv_j(y^{-},x)}p_j(x|y_{j-1})d\kappa_j(x)
=\int_{\cX_{j+n,\infty}} e^{itv_{j+n}(y^{-},\cdot)}\cL_{j,t}^n(e^{it v_j(y^{-},\cdot)}p_j(\cdot|y_{j-1}))d\kappa_{j+n}
$$
Using Assumption \ref{CondAss} and Remark \ref{Remmm} we see that 
\begin{equation}\label{CharEst}
 |\ka_j(e^{itS_{j,n}M})|\leq C\|\cL_{j,t}^n\|_{a^{1/3},\om^{1/3}}\to 0\,\,\text{ as }n\to\infty
\end{equation}
where $\|\cL_{j,t}^n\|_{a^{1/3},\om^{1/3}}$ is the operator norm with respect to norms $\|\cdot\|_{j,a^{1/3},\om^{1/3}}$ on the domain and  $\|\cdot\|_{j+n,a^{1/3},\om^{1/3}}$ on the range.
Therefore
$\lim_{n\to\infty} \mu_j(e^{itS_{j,n}M})=0$
  and the proof of the proposition is complete.
 \end{proof}

\subsection{Local limit theorem in the irreducible case}
Let us take a zero mean sequence of functions $f_j:\cY_j\to\bbR$ such that $\sup_j\|f_j\|_{j,\om}<\infty$ and define 
$$
S_nf=\sum_{j=0}^{n-1}f_j(...,X_{j-1},X_j,X_{j+1},...)
$$
By
 Lemma \ref{Sinai} there are sequences of functions $g_j:\cX_{j,\infty}\to\bbR$ and 
 $u_j: \cY_j\to\bbR$ such that 
 $\sup_j\|g_j\|_{j,a^{1/3},\om^{1/3}}<\infty$,  $\sup_j\|u_j\|_{j,a^{1/3},\om^{1/3}}<\infty$ and 
 $
f_j=g_j+u_{j+1}-u_j.
 $
Let $\cL_{j}$ be the transfer operators corresponding to $\ka_j$ and for every $t\in\bbR$ let $\cL_{j,t}(g)=\cL_j(ge^{it \phi_j})$.

 As it was explained in \S \ref{SSInt-LLT},
 the non-lattice LLT 
and the first order expansions follow from  the  two results below.

\begin{lemma}\label{estt1}
There are constants $\del_0,C_0,c_0\!\!>\!\!0$ such that for every $t\!\in\![-\del_0,\del_0]$ we have
$$
|\mu_0(e^{itS_n g})|\leq C_0e^{-c_0\sig_n^2 t^2}
$$
where $\sig_n=\|S_n f\|_{L^2}$.
\end{lemma}

\begin{lemma}\label{estt2}
Let $\del_0$ be like in Lemma \ref{estt1}. 
If $(f_j)$ is irreducible then for every $T>\del_0$ we have 
$
 \int_{\del_0\leq |t|\leq T} 
|\gamma_0(e^{it S_nf})|dt=o(\sig_n^{-1}).
$
\end{lemma}

\begin{proof}[Proof of Lemma \ref{estt1}]
Arguing like in the proof of Proposition \ref{Red2sided}, we see that
there is a constant $C>0$ such that for all $t$ and $n$ we have
\begin{equation}\label{CharEst}
 |\mu_0(e^{itS_n f})|\leq C\|\cL_{0,t}^n\|_{a^{1/3},\omega^{1/3}}.
\end{equation}
Now the result follows from the corresponding result in the one sided case, noting that $\|S_n g\|_{L^2}=\|S_n f\|_{L^2}+O(1)$.
\end{proof}

\begin{proof}[Proof of Lemma \ref{estt2}]
By Proposition \ref{Red2sided} 
the sequence of functions $(g_j)$ is also irreducible. 
Thus, the lemma follows from \eqref{CharEst} and 
 \eqref{Suff J} which holds in the irreducible case.
\end{proof}

\subsection{LLT in the reducible case}
Using Lemma \ref{Sinai} and conditioning on the past, in the reducible case we can also prove an LLT similar to the one in Section~\ref{Red sec}. 
By Proposition \ref{Red2sided} we have $R(\phi)=R(\psi)$. In particular $a(\phi)=a(\psi)$, where $a(\cdot)$ was defined at the beginning of Section \ref{Red sec}. 
Now, 
the decay of the characteristic functions at the relevant points needed in the proof of Theorem \ref{LLT RED} can be obtained by repeating the arguments in the irreducible case. The second ingredient is to expand the characteristic functions around points in $(2\pi/a(g))\bbZ$. This is done by conditioning on the past and using an appropriate Perron-Frobenius theorem 
(see \cite[Theorem D.2]{DolgHaf PTRF 2}) for each realization on the past, and then integrating. In order not to overload the paper the exact details are left for the reader.

\section{Applications: revisiting some of the examples in \cite{MarShif1}}
In \cite[Section 4]{MarShif1} we provided many examples of processes which fit into the weaker type of sequential functional analytic Markovian framework considered there. In this section we will show that we can also fit a modified version of some of these examples into the framework of this paper. The main diffuclity here is to overcome the fact that we require stronger ``variation" than the approximation coefficients $v_{j,p,\delta}(g)=\sup_r\delta^{-r}\|g-\bbE[g|\cF_{j-r,j+r}]$ considered in \cite{MarShif1}. This will be done by assuming the objects defining the examples are equicontinuous. 
We note that we do not include here applications to dynamical systems since our approach does not yield new examples compared with \cite{DolgHaf LLT} (although we can cover some of the results there using our approach).
\subsection{Products of positive matrices}
Let us take some $d>1$ and $C>1$ and denote by $M(d,C)$ the set of all matrices with entries in $[C^{-1},C]$.
Let $(X_j)_{j\in\bbZ}$. Let $\cX_j$ be the state space of $X_j$. Let $A_j:\cX_j\to M(d,C)$ be a family of equicontinuous  functions. Let us extend $A_j$ to a two sided sequence by setting $(A_j)_{k,\ell}(X_j)=1$ for all $j<0$ and $1\leq k,\ell\leq d$.  
 Define 
$$
\textbf{A}_j^n(x_j,...,x_{j+n-1})=A_{j+n-1}(x_{j+n-1})\cdots A_{j+1}(x_{j+1})A_j(x_j).
$$
Then by the arguments of \cite[Ch.4]{HK}
we obtain the following random Perron Frobenius theorem. There are functions $\la_j=\la_j(x_j,x_{j+1},...)$ positive vectors $h_j=h_j(...,x_{j-1},x_j,x_{j+1},...)$, vectors $\nu_j=\nu_j(...,x_{j-1},x_j,x_{j+1},...)$ and constants $C>0$ and $\delta\in(0,1)$ such that, for all $j$ and $n$, 
\begin{equation}\label{RRPF}
\left\|\frac{\textbf{A}_j^n}{\la_{j,n}}-\nu_j\otimes h_{j+n}\right\|\leq C\delta^n    
\end{equation}
where $\la_{j,n}=\prod_{k=j}^{j+n-1}\la_k$ and $(\nu_j\otimes h_{j+n})(g)=\nu_j(g)h_{j+n}$. Moreover, $A_j(x_j)h_j=\la_jh_{j+1}, (A_j(x_j))^*\nu_{j+1}=\la_j\nu_j$ and $\nu_j(h_j)=1$. 
Now, by \cite[Eq. (5.10)]{Nonlin} there are positive linear functionals $\mu$ such that 
$$
\la_j(x_j,x_{j+1},...)=\lim_{n\to\infty}F_{j,n}(x_j,x_{j+1},...,x_{j+n})
$$ 
uniformly in $j$,
where 
$$
F_{j,n}(x_j,x_{j+1},...,x_{j+n})=\frac{\mu(A_{j+n}(x_{j+n})\cdots A_j(x_j)\textbf{1})}{\mu(A_{j+n}(x_{j+n})\cdots A_{j+1}(x_{j+1})\textbf{1})}.
$$
Now, it is not hard to show that for a fixed $n$ the family of functions $F_{j,n}, j\in\bbZ$ are equicontinuous. Thus, by taking the limit as $n\to\infty$ we conclude that $\la_j$ are are equicontinuous as well. 
 Therefore, since $\la_j$ is bounded and bounded away from $0$ we can get limit theorems for
$$
\ln\|\textbf{A}_0^n(X_0,X_{1},...,X_{n-1})\|\text{ and }\ln\left(\nu(\textbf{A}_0^n(X_0,X_{1},...,X_{n-1})g)\right)
$$
where $g$ and $\nu$ are two vectors with positive entries. Indeed, the entire problem reduces to the partial sums $\sum_{j=0}^{n-1}\ln(\lambda_j(X_0,...,X_j,X_{j+1},...))$.

\subsection{Random Lyapunov exponents}\label{Lyp Sec}
Let $d>1$ and let $A$ be a hyperbolic matrix with distinct eigenvalues $\la_1,...,\la_d$. Suppose that for some $k<d$ we have $\la_1<\la_2<...<\la_k<1<\la_{k+1}<...<\la_{d}$. Let $h_j$ be the corresponding eigenvalues.

Now, let $(A_j)$ be a sequence of matrices such that $\sup_j \|A_j-A\|\leq \ve$. Then, if $\ve$ is small enough there are numbers $\la_{j,1}<\la_{j,2}<...<\la_{j,k}<1<\la_{j, k+1}<...<\la_{j,d}$
and vectors $h_{j,i}$ such that 
$$
A_{j}h_{j,i}=\la_{j,i}h_{j+1,i}.
$$
Moreover, $\sup_j|\la_{j,i}-\la_i|$ and 
$\sup_j\|h_{j,i}-h_{i}\|$ converge to $0$ as $\ve\to 0$.

Now, the sequence $(A_j)$ is uniformly hyperbolic and the sequences 
$(\la_{1,j})_j,...,(\la_{d,j})_j$ can be viewed as its sequential Lyapunov exponents. Moreover,  
the one dimensional  spaces $H_{i,j}=\text{span}\{h_{i,j}\}$ can be viewed as its sequential Lyapunov spaces.
Next,  $\la_{i,j}$ and $h_{i,j}$ 
can be approximated exponentially fast in $n$ by functions of 
$$(A_{j-n},A_{j-n+1},...,A_j,A_{j+1},...,A_{j+n}),$$ uniformly in $j$. 

Finally, let us consider a Markov chain $(X_j)$ and let us take random  matrices $A_j=A_j(X_j)$ such that 
$$
\sup_{j}\|A_j-A\|_{L^\infty}\leq \varepsilon
$$
and $x_j\to A_j(x_j)$ is an equicontinuous family of functions.
Then if $\varepsilon$ is small enough the random variables $\la_{i,j}$ and $h_{i,j}$ can be approximated exponentially fast by equicontinuous  functions of  $X_{j+k}, |k|\leq r$ as $r\to\infty$,  and thus the resulting limits are also  equicontinuous.
\subsection{Linear processes}
Let $g_{j}:\cX_j\to\bbR$ be measurable uniformly bounded functions, let $(a_k)$ be a sequence of numbers and define 
$$
f_j(...,X_{j-1},X_j,X_{j+1},...)=\sum_{k\in\bbZ}a_kg_{j-k}(X_{j-k}),
$$
assuming that the above series converges. Note that 
$$
\|f_j\|_{L^\infty}\leq \sum_{k}|a_k|\|g_{j-k}\|_{L^\infty}
$$
and so $f_j$ are uniformly bounded. Next,
$$
|f_j(x)-f_j(y)|\leq C\sup_m|g_m(x_{m})-g_m(y_m)|.
$$
Therefore, if we assume that the functions $g_m$ are equicontinuous then also $f_j$ are equicontinuous.

\subsection{Iterated random functions driven by inhomogeneous  Markov chains}\label{Iter}
 Iterated random functions (cf. \cite{[15]}) are an important class of processes. Many nonlinear models like GARCH,
bilinear and threshold autoregressive models fit into this framework.
Here we describe a non-stationary version of such processes which are based on Markov chains $X_j$ instead of iid sequences.
Beyond the theoretical interest, considering such processes driven by inhomogeneous Markov chains rather than iid variables could be useful for practitioners.

\subsubsection{Iterated functions with initial condition}
Let $G_k:\bbR\times\cX_k\to\bbR$ be measurable functions. Let
$L_k(x_k)$ denote the Lipschitz  constant of the function $G_k(\cdot,x_k)$.
Define a process recursively by setting $Y_0=y_0$ to be a constant then  setting $Y_k=G_k(Y_{k-1},X_k), k\geq 1$. 
 Notice that for $k\geq1$,
$$
Y_k=G_{k,X_k}\circ G_{k-1,X_{k-1}}\circ\cdots\circ G_{1,X_1}(y_0):=f_k(X_1,...,X_{k-1},X_k).
$$
where $G_{s,X_s}(y)=G_s(y,X_s)$. This fits our model of functions $f_k$ that depend on the entire path of a two sided Markov chain $(X_j)_{j\in\bbN}$ (note that one can always extend $X_j$ to a two sided sequence simple by considering iid copies of $X_0$, say, which are also independent of $X_j, j\geq0$). Namely, by abusing the notation we may write 
$$
f_k(X_1,...,X_{k-1},X_k)=f_k(...,X_{k-1},X_{k},X_{k+1},...)
$$
where the dependence is only on $X_1,...,X_{k-1},X_k$.
\begin{lemma}
Suppose that there are uniformly bounded sets $K_i(x_i)\subset\bbR$ such that 
 $G_{i,X_i}(K_i(X_i))\subset K_{i+1}(X_{i+1})$ almost surely for all $i$. Assume also that $K_1:=\cap_{x_1}K_1(x_1)\not=\emptyset$. Then if we start with   $y_0\in K_1$ then 
 $$
\sup_k\|Y_k\|_{L^\infty}<\infty.
 $$
 In particular this is the case when the functions $G_{i}$ are uniformly bounded.
\end{lemma}

\begin{lemma}
Suppose that $|G_{k,x}(a)|\leq C|a|$ for all $x,a$ and $k$ for some $C>0$, that
$\delta_0:=\sup_k\|L_k(X_k)\|_{L^\infty}<1$ and  $\left|G_{k,x}(a)-G_{k,z}(a)\right|\leq C\om(d_k(x,z))|a|$ for some modulus of continuity $\om$, $C>0$ and all $a$. Then the functions $f_j$ are equicontinuous with respect to the variation $\om_j(x,y)=\sup_n\delta^n\omega(d_{j+n}(x_{j+n},y_{j+n}))$, for all $\delta_0<\delta<1$.
\end{lemma}
\begin{proof}
We have 
$$
f_k(x_1,...,x_k)-f_k(z_1,...,z_k)=\sum_{s=0}^kG_{k,x_k}\circ\cdots G_{s+1,x_{s+1}}\circ (G_{s,x_s}-G_{s,z_s})\circ G_{s-1,z_{s-1}}\circ\cdots\circ G_{1,y_1}(y_0). 
$$
Now,
$$
\left| G_{s-1,y_{s-1}}\circ\cdots\circ G_{1,y_1}(y_0)\right|=\left|G_{s-1,y_{s-1}}\circ\cdots\circ G_{1,y_1}(y_0)-G_{s-1,y_{s-1}}\circ\cdots\circ G_{1,y_1}(0)\right|\leq \delta_0^{s-1}
$$
and so
$$
\left|(G_{s,x_s}-G_{s,y_s})\circ G_{s-1,y_{s-1}}\circ\cdots\circ G_{1,y_1}(y_0)\right|\leq C\omega(d_s(x_s,y_s))\delta_0^{s-1}|y_0|.
$$
Therefore,
$$
\left|f_k(x_1,...,x_k)-f_k(y_1,...,y_k)\right|\leq C\max_{s\leq k}\delta^s\omega(d_s(x_s,y_s)).
$$
\end{proof}
Let us discuss the special conditions $|G_{k,x}(a)|\leq C|a|$ and $\left|G_{k,x}(a)-G_{k,y}(a)\right|\leq \om(d_k(x,y))|a|$. Consider, for instance, the case when
$G_k(a,x)=\sum_{\ell=1}^d H_{k,\ell}(a)F_{k,\ell}(x)$ (such functions are dense in an appropriate sense). Then the condition holds true if the functions $H_{k,\ell}$ and $F_{k,\ell}$ are uniformly bounded,  $|H_{k,\ell}(a)|\leq C|a|$ (i.e., $H_{k,\ell}(0)=0$ and $\sup_{k,\ell,z}|H_{k,\ell}'(z)|<\infty$) and $\sup_{k,\ell}|F_{k,\ell}(x)-F_{k,\ell}(y)|\leq \omega(|x-y|)$. In particular we can consider the case when the derivatives of  $F_{k,\ell}$ are uniformly bounded, assuming that $\cX_j$ are manifolds.




\subsubsection{Iterated functions without initial condition}
The processes $Y_k$ defined in the previous section can never be stationary since $Y_0=y_0$ and $Y_k$ depends on $X_1,...,X_k$.  
Here we considering a related class of recursive sequences  which will be stationary when the chain $(X_j)$ is stationary and the functions $G_k$ coincide.  
Let $G_k:\bbR\times\cX_k\to\bbR$ be like in the previous section. We define recursively $Y_k=G_k(Y_{k-1},X_k)=G_{k,X_k}(Y_{k-1})$. Then there is a measurable functions $f_k$ on $\prod_{j\leq k}\cX_j$ such that 
$$
Y_k=f_k(...,X_{k-1},X_k).
$$
This fits our general framework by either considering functions $f_k$ which depend only the the coordinates $x_j, j\leq k$.  Then like in the previous section we have the following two results.

\begin{lemma}
Suppose that there are uniformly bounded sets $K_i(x_i)\subset\bbR$ such that 
 $G_{i,X_i}(K_i(X_i))\subset K_{i+1}(X_{i+1})$ almost surely for all $i$. Assume also that $K_1:=\cap_{x_1}K_1(x_1)\not=\emptyset$. Then if we start with   $y_0\in K_1$ then 
 $$
\sup_k\|Y_k\|_{L^\infty}<\infty.
 $$
 In particular this is the case when the functions $G_{i}$ are uniformly bounded.
\end{lemma}

\begin{lemma}
Suppose that $|G_{k,x}(a)|\leq C|a|$ for all $x,a$ and $k$ for some $C>0$, that
$\delta_0:=\sup_k\|L_k(X_k)\|_{L^\infty}<1$ and  $\left|G_{k,x}(a)-G_{k,z}(a)\right|\leq C\om(d_k(x,z))|a|$ for some modulus of continuity $\om$, $C>0$ and all $a$. Then the functions $f_j$ are equicontinuous with respect to the variation $\om_j(x,y)=\sup_n\delta^n\omega(d_{j+n}(x_{j+n},y_{j+n}))$, for all $\delta_0<\delta<1$.
\end{lemma}



\end{document}